\documentclass[12pt,oneside, dvipsnames]{amsart}
\usepackage{hyperref}
\hypersetup{
    colorlinks=true,
    linkcolor=black,
    citecolor=magenta,
    filecolor=black,      
    urlcolor=blue,
    }
\usepackage{tkz-euclide,fullpage}
\tikzstyle{directed}=[postaction={decorate,decoration={markings,
    mark=at position .65 with {\arrow{stealth}}}}]
\tikzstyle{reverse directed}=[postaction={decorate,decoration={markings,
    mark=at position .5 with {\arrowreversed{stealth};}}}]
\usetikzlibrary{hobby}
\usetikzlibrary{decorations.pathreplacing}
\usepackage{amssymb}
\usepackage{latexsym}
\usepackage{amsmath}
\usepackage{mathrsfs}
\usepackage[all]{xy}
\usepackage{enumerate}
\usepackage{amsthm}
\usepackage{psfrag}
\usepackage{ifthen}
\usepackage{mathdots}
\usepackage{caption}
\usepackage{subcaption}
\usepackage{standalone}
\usepackage{mathtools}
\usepackage{pgfplots}
\usepackage{xfrac}
\pgfplotsset{compat=1.11}
\usepackage[mathscr]{euscript}
\usepackage{bm}
\usepackage{amsbsy}
\usepackage[outline]{contour}
\usepackage{todonotes}
\usepackage{url}
\usepackage{bbm}
\usepackage{tikz-3dplot}

\newcommand{\newbold}[1]{\bgroup\contourlength{0.01em}\contour{black}{#1}\egroup}

\graphicspath{{graphics/}}

\newtheorem{theorem}{Theorem}[section]
\newtheorem{lemma}[theorem]{Lemma}
\newtheorem{metalemma}[theorem]{Meta Lemma}
\newtheorem{proposition}[theorem]{Proposition}
\newtheorem{corollary}[theorem]{Corollary}
\newtheorem{problem}[theorem]{Problem}

\theoremstyle{definition}
\newtheorem{definition}[theorem]{Definition}

\newtheorem{remark}{Remark}

\newtheorem{example}[theorem]{Example}

\newenvironment{cor}{\begin{corollary}}{
\end{corollary}}

\usepackage{tabularx,environ}

\makeatletter
\newcommand{\problemtitle}[1]{\gdef\@problemtitle{#1}}
\newcommand{\probleminput}[1]{\gdef\@probleminput{#1}}
\newcommand{\problemquestion}[1]{\gdef\@problemquestion{#1}}
\NewEnviron{algproblem}{
  \problemtitle{}\probleminput{}\problemquestion{}
  \BODY
  \par\addvspace{.5\baselineskip}
  \noindent
  \begin{tabularx}{\textwidth}{@{\hspace{0em}} l X c}
    \hline\hline
    \multicolumn{2}{@{\hspace{0em}}l}{\@problemtitle} \\
    \textbf{In:} & \@probleminput \\%
    \textbf{Question:} & \@problemquestion\\
    \hline\hline
  \end{tabularx}
  \par\addvspace{.5\baselineskip}
}
\makeatother

\makeatletter
\NewEnviron{algproblemnl}{
  \problemtitle{}\probleminput{}\problemquestion{}
  \BODY
  \par\addvspace{.5\baselineskip}
  \noindent
  \begin{tabularx}{\textwidth}{@{\hspace{0em}} l X c}
    \multicolumn{2}{@{\hspace{0em}}l}{\@problemtitle} \\
    \textbf{In:} & \@probleminput \\%
    \textbf{Question:} & \@problemquestion\\
  \end{tabularx}
  \par\addvspace{.5\baselineskip}
}
\makeatother

\newcommand{\comment}[1]
	   {\ifthenelse{\equal{\showcomments}{yes}}
	     {\footnotemark\marginpar{\sffamily{\tiny
		   \addtocounter{footnote}{-1}\footnotemark#1
}\normalfont}}{}}

\newcommand{\showcomments}{yes}

\newcommand{\bel}[1]{\begin{equation}\label{#1}}
\newcommand{\ee}{\end{equation}}

\newcommand{\rank}{\text{rank}}

\newcommand{\LBA}{\left\{\begin{array}}
\newcommand{\EAR}{\end{array}\right.}

\def\ovb{{\overline{b}}}
\def\ovc{{\overline{c}}}
\def\ovd{{\overline{d}}}
\def\ove{{\overline{e}}}
\def\ovf{{\overline{f}}}
\usepackage{bm}

\def\ovo{{\overline{0}}}

\def\ovr{{\overline{r}}}
\def\ovs{{\overline{s}}}
\def\ovu{{\overline{u}}}
\def\ovv{{\overline{v}}}
\def\ovx{{\overline{x}}}

\def\ovdelta{{\overline{\delta}}}
\def\ovfdelta{{\overline{f}{}^{\ovdelta}}}

\def\CC{{\mathcal C}}
\def\CE{{\mathcal E}}

\def\CL{{\mathcal L}}

\def\P{{\mathbf{P}}}
\def\NP{{\mathbf{NP}}}
\def\DP{{\mathbf{DP}}}
\def\bv{{\mathbf{v}}}

\def\QSP{{\textsf{QSP}}}
\def\ZOE{{\textsf{ZOE}}}
\def\TPART{{\textsf{3PART}}}

\newcommand{\gp}[1]{{\left\langle #1 \right\rangle}}

\newcommand{\rb}[1]{{\left( #1 \right)}}

\newcommand{\Set}[2]{\left\{\, #1 \;\middle|\; #2 \,\right\}}

\def\MN{{\mathbb{N}}}
\def\MZ{{\mathbb{Z}}}
\def\MQ{{\mathbb{Q}}}
\def\MR{{\mathbb{R}}}

\DeclareMathOperator{\Mat}{{Mat}}
\DeclareMathOperator{\Aut}{{Aut}}
\DeclareMathOperator{\Span}{{Span}}

\DeclareMathOperator{\im}{{im}}

\DeclareMathOperator{\size}{{size}}
\DeclareMathOperator{\supp}{{supp}}

\DeclareMathOperator{\diam}{{diam}}
\DeclareMathOperator{\spn}{{span}}

\DeclareMathOperator{\gr}{{gr}}

\newcommand{\sasha}[1]{\todo[inline,color=cyan]{\footnotesize SU: #1}}

\usepackage[framemethod=TikZ]{mdframed}
\usepackage[listings,theorems,most]{tcolorbox}
\tcbset{before={\par\medskip\pagebreak[0]\noindent},after={\par\medskip}}
\usepackage{xcolor}

\makeatletter
\renewcommand{\tocsection}[3]{%
  \indentlabel{\@ifnotempty{#2}{\bfseries\ignorespaces#1 #2\quad}}\bfseries#3}
\renewcommand{\tocsubsection}[3]{%
  \indentlabel{\@ifnotempty{#2}{\ignorespaces#1 #2\quad}}#3}

\newcommand\@dotsep{4.5}
\def\@tocline#1#2#3#4#5#6#7{\relax
  \ifnum #1>\c@tocdepth 
  \else
    \par \addpenalty\@secpenalty\addvspace{#2}%
    \begingroup \hyphenpenalty\@M
    \@ifempty{#4}{%
      \@tempdima\csname r@tocindent\number#1\endcsname\relax
    }{%
      \@tempdima#4\relax
    }%
    \parindent\z@ \leftskip#3\relax \advance\leftskip\@tempdima\relax
    \rightskip\@pnumwidth plus1em \parfillskip-\@pnumwidth
    #5\leavevmode\hskip-\@tempdima{#6}\nobreak
    \leaders\hbox{$\m@th\mkern \@dotsep mu\hbox{.}\mkern \@dotsep mu$}\hfill
    \nobreak
    \hbox to\@pnumwidth{\@tocpagenum{\ifnum#1=1\bfseries\fi#7}}\par
    \nobreak
    \endgroup
  \fi}
\AtBeginDocument{%
\expandafter\renewcommand\csname r@tocindent0\endcsname{0pt}
}
\def\l@subsection{\@tocline{2}{0pt}{2.5pc}{5pc}{}}
\makeatother

\title{Orientable quadratic equations in wreath products}
\author{Alexander Ushakov}
\author{Chloe Weiers}
\date{\today}

\begin{document}

\maketitle

\begin{abstract}
In this paper we study the complexity of solving orientable quadratic equations in wreath products $A\wr B$ of finitely generated abelian groups. We give a classification of cases 
(depending on genus and other characteristics of a given equation)
when the problem is computationally hard or feasible.
\\
\noindent
\textbf{Keywords.}
Diophantine problem, quadratic equations, spherical equations, metabelian groups, wreath product, complexity, NP-completeness.

\noindent
\textbf{2010 Mathematics Subject Classification.} 
20F16, 20F10, 68W30.
\end{abstract}

\section{Introduction}

Let $F = F(Z)$ denote the free group on countably many generators $Z = \{z_i\}_{i=1}^\infty$. For a group $G$, an \emph{equation over $G$ with variables in $Z$} is an equation of the form $W = 1$, where $W \in F*G$. If $W = z_{i_1}g_1\cdots z_{i_k}g_k$, with $z_{i_j}\in Z$ and $g_j\in G$, then we refer to $\{z_{i_1},\ldots,z_{i_k}\}$ as the set of \emph{variables} and to $\{g_1,\ldots,g_k\}$ as the set of \emph{constants} (or \emph{coefficients}) of $W$. We occasionally write $W(z_1,\ldots,z_k)$ or $W(z_1,\ldots,z_k;g_1,\ldots,g_k)$ to indicate that the variables in $W$ are precisely $z_1,\ldots,z_k$ and (in the latter case) the constants are precisely $g_1,\ldots, g_k$. 

A \emph{solution} to an equation $W(z_1,\ldots,z_k)=1$ over $G$ is a homomorphism $$\varphi\colon F*G\rightarrow G $$
such that $\varphi|_G = 1_G$ and $W\in \ker \varphi$. If $\varphi$ is a solution of $W=1$ and $g_i = \varphi(z_i)$, then we often say that $g_1,\ldots,g_k$ is a solution of $W=1$. We also write $W(g_1,\ldots,g_k)$ or $W(\bar{g})$ to denote the image of $W(z_1,\ldots,z_k)$ under the homomorphism $F*G \rightarrow G$ which maps $z_i \mapsto g_i$ (and restricts to the identity on $G$). Note that while some authors allow for solutions in some overgroup $H$ (with an embedding $G\hookrightarrow H$), in this article we only consider solutions in $G$. 

In this paper we assume that $G$ comes equipped with a fixed generating set $X$ and elements of $G$ are given as products of elements of $X$ and their inverses. This naturally defines the length (or size) of the equation $W=1$ as the length of its left-hand side $W$.

\begin{definition}
An equation $W=1$ is called \emph{quadratic} if each variable appears exactly twice (as either $z_i$ or $z_i^{-1}$).
\end{definition}

The \emph{Diophantine problem} ($\DP$) in a group $G$ for a class of equations $C$ is an algorithmic question to decide whether a given equation $W=1$ in $C$ has a solution. In this paper we study the class of quadratic equations over wreath products of finitely generated abelian groups.

\subsection{Classification of quadratic equations}
We say that equations $W = 1$ and $V=1$ are \emph{equivalent} 
if there is an automorphism $\phi\in \Aut(F\ast G)$ such that 
$\phi$ is the identity on $G$ and $\phi(W) = V$. It is a well 
known consequence of the classification of compact surfaces 
that any quadratic equation over $G$ is equivalent, via an 
automorphism $\phi$ computable in time $O(|W|^2)$, to an equation 
in exactly one of the following three \emph{standard forms} (see 
\cite{Comerford_Edmunds:1981,Grigorchuk-Kurchanov:1992}):
\begin{align}
\prod_{j=1}^m z_j^{-1} c_j z_j&=1 &m\ge 1,\label{eq:spherical}\\
\prod_{i=1}^g[x_i,y_i]\prod_{j=1}^m z_j^{-1} c_j z_j&=1 &g\geq 1, m\geq 0, \label{eq:orientable}\\
\prod_{i = 1}^g x_i^2\prod_{j=1}^m z_j^{-1} c_j z_j&=1 &g \geq 1, m \geq 0.\label{eq:nonorientable}
\end{align}
The number $g$ is the \emph{genus} of the equation, and both $g$ and $m$ 
(the number of constants) are invariants. 
The standard forms are called, respectively, 
\emph{spherical}, \emph{orientable of genus $g$}, and 
\emph{non-orientable of genus $g$}.

\subsection{Previous results}
The Diophantine problem for 
quadratic equations naturally generalizes fundamental (Dehn) problems of group theory such as the word and conjugacy problems. There is a profound relationship between quadratic equations and compact surfaces, which makes quadratic equations an interesting
object of study.

The study of quadratic equations originated with Malcev,
who in \cite{Malcev:1962} considered some particular
quadratic equations over free groups. This line of research for free groups was continued with the study of solution sets in \cite{Grigorchuk-Kurchanov:1992}. Furthermore, $\NP$-completeness was proved in \cite{Diekert-Robson:1999,Kharlampovich-Lysenok-Myasnikov-Touikan:2010}.
Various classes of infinite groups were analyzed:
hyperbolic groups were studied in
\cite{Grigorchuk-Lysenok:1992,Kharlampovich-Taam:2017},
the first Grigorchuk group was studied in
\cite{Lysenok-Miasnikov-Ushakov:2016} and \cite{Bartholdi-Groth-Lysenok:2022},
free metabelian groups were studied in
\cite{Lysenok-Ushakov:2015,Lysenok-Ushakov:2021},
and metabelian Baumslag--Solitar groups were studied in
\cite{Mandel-Ushakov:2023b}.
The Diophantine problem for quadratic equations over the 
lamplighter group was recently shown to be decidable by Kharlampovich, Lopez, 
and Miasnikov 
(see \cite{Kharlampovich-Lopez-Miasnikov:2020}). Additionally, Ushakov and Weiers proved that the Diophantine problem for orientable quadratic equations over the lamplighter group is decidable in polynomial time, and the problem for non-orientable equations of genus two is decidable in linear time and of genus one is $\NP$-complete (see \cite{ushakov2023quadratic}).

Natural applications of quadratic equations (particularly spherical equations) 
to cryptography have also been explored. In particular, 
\cite{ushakov2024constrained} establishes connections
between computational group theory and lattice-based cryptography 
via spherical equations over a certain class of finite metabelian groups.

\subsection{Our results}

This paper goes beyond the lamplighter group $\MZ_2\wr \MZ$ studied in
\cite{ushakov2023quadratic} and provides a deep analysis 
of the computational combinatorics for orientable quadratic equations
over wreath products $A\wr B$ of finitely generated abelian groups.
In \cite{ushakov2023quadratic}, it was shown that
the number of conjugates $m$ is the source of hardness for solving 
equations \eqref{eq:spherical} in the lamplighter group.
For a wreath product of general finitely generated abelian groups
$A\wr B$ the situation is more complicated.
Throughout the paper we assume that $A$ is not trivial.

First, we show that for a general wreath product of finitely generated abelian groups $A$ and $B$, the Diophantine problem for orientable equations remains in $\NP$, even when 
$A$ and $B$ are given as part of the input (the uniform problem). 
However, in the general case, parameters such as the ranks of 
$A$ and $B$ and the genus of the equation impact the complexity 
of solving these equations.
\begin{itemize}
\item 
The complexity of the problem is directly influenced by the group $B$.
Specifically, for a fixed $A\wr B$,
the problem is $\NP$-hard if and only if
$|B|=\infty$, see Corollary \ref{co:DP-infinite-B}.
\item 
The contribution of the group $A$ to the complexity is limited in a sense
that the problem is $\NP$-hard for $A_1\wr B$
if and only if it is $\NP$-hard for $A_2\wr B$,
see Corollary \ref{co:irrelevant-A2}.
However, if $A$ is a part of the input, the problem becomes $\NP$-hard
even when $B=\MZ_2$ is fixed, see Corollary \ref{co:finite-B}.
\item 
The genus $g$ also plays a crucial role in the complexity of the equation.
If $g$ is sufficiently large (when $g\ge \rank(B)/2$), then
the equation can be solved in polynomial time, see Corollary \ref{co:g-vs-rankB}.
Otherwise, for intermediate values of $g$ such as $g < \rank(B)/2$, 
the problem generally remains $\NP$-hard as shown in
Corollary \ref{co:mid-h}.
\end{itemize}

\subsection{Model of computation and representation of basic objects}

We use a random-access machine as a model of computation.
Integers are given in unary form unless stated otherwise.
Elements of $A\wr B$ are given as words, i.e., sequences of letters
of the group alphabet $\{a_1^{\pm1},\ldots,a_n^{\pm1},b_1^{\pm1},\ldots,b_m^{\pm1}\}$ for $A\wr B$. Additional details for representation of certain objects are given as needed later in the paper.


\subsection{Outline}
In Section \ref{se:Preliminaries}, we recall the definition of the wreath product
$A\wr B$ and review its basic properties. We also discuss the properties of finitely generated abelian groups and lattices, 
and remind the reader of the definitions of the three partition problem and 
the zero-one equation problem.
In Section \ref{se:QSP}, we define a decision problem 
termed the quotient-sum problem ($\QSP$), to which we later reduce the Diophantine problem for orientable quadratic equations. In Section \ref{se:QSP-properties},
we examine the computational properties of several variations of $\QSP$, which correspond to the Diophantine problem for certain natural classes of equations over wreath products $A\wr B$.
In Section \ref{se:orientable}, 
we use the computational properties of $\QSP$ established in Section \ref{se:QSP-properties}
to draw conclusions about orientable quadratic equations.
The paper concludes with Section \ref{se:Conclusion}.


\section{Preliminaries}
\label{se:Preliminaries}

\subsection{Wreath products of abelian groups}

We begin by reviewing basic facts about restricted wreath products. 
Let $A$ and $B$ be groups. Define a set 
$$
A^B = \Set{f\colon B\to A}{|\supp(f)|<\infty}
$$
and a binary operation $+$ on $A^B$ which, for $f,g\in A^B$, produces $f+g\in A^B$ defined by
$$
(f+g)(x) = f(x)+g(x) \ \mbox{ for } \ x\in B.
$$
For $f\in A^B$ and $b\in B$, define $f^b\in A^B$ by
$$
f^b(x)=f(bx) \
\mbox{ for } \ x\in B,
$$
which is a right $B$-action on $A^B$ because $f^e=f$ (where $e\in B$ is the identity of $B$) and $f^{(b_1b_2)}(x)=f(b_1b_2x)=
(f^{b_1})^{b_2}(x)$. 
Hence, we can consider a semidirect product $B \ltimes A^B$
equipped with the operation
$$
(\delta_1,f_1)(\delta_2,f_2)=
(\delta_1+\delta_2,f_1^{\delta_2}+f_2).
$$
The group $B \ltimes A^B$ is called the \emph{restricted wreath product} of $A$ and $B$ and is denoted by $A\wr B$. Henceforth we assume that $A$ and $B$ are finitely generated abelian groups and use additive group notation.

It is convenient to treat $a\in A$ as a function $f\in A^B$ defined as
$$
f(x)=
\begin{cases}
a&\mbox{ if }x=0,\\
0&\mbox{ otherwise.}
\end{cases}
$$
For $a\in A$ and $b\in B$, that allows us to use the notation $a^b$
for a function $f\in A^B$ such that
$$
f(x)=
\begin{cases}
a&\mbox{ if }x=-b,\\
0&\mbox{ otherwise,}
\end{cases}
$$
and, slightly abusing notation, 
the corresponding element $(0,f)\in A\wr B$.
Then every $f\in A^B$ can be expressed as 
\[
f = \sum_{i=1}^k a_i^{-b_i},
\]
if $\supp(f) = \{b_1,\ldots, b_k\}$ and $f(b_i)=a_i$.
This notation is unique for $f$ if the exponents $b_1,\ldots,b_k$ are distinct
and $a_1,\dots,a_k$ are nontrivial.

\subsection{$A^B$ as a group ring}

For a commutative ring $R$ with unity and a group 
$(G,\cdot)$ we can define the \textit{group ring} $R[G]$ of $G$ with 
scalar coefficients in $R$ as the set of linear combinations 
of elements of $G$
$$
R[G]=
\Set{\sum_{g\in G} r_g\cdot g}{r_g\in R, \mbox{ where all but finitely many $r_g$ are $0$}},
$$
with addition defined by
$$
\sum_{g\in G} r_g\cdot g+
\sum_{g\in G} s_g\cdot g=
\sum_{g\in G} (r_g+s_g)\cdot g
$$
and multiplication $\times$ defined by
$$
\rb{\sum_{g\in G} r_g\cdot g}
\times
\rb{\sum_{g\in G} s_g\cdot g}=
\sum_{g\in G} t_g\cdot g,
\mbox{ where } 
t_g=\sum_{g_1\cdot g_2=g}r_{g_1}s_{g_2}.
$$
Now, by the fundamental theorem of finitely generated abelian groups,
$$
A\ \simeq\ \MZ^s\times \MZ_{\alpha_1}\times\dots\times\MZ_{\alpha_t}.
$$
This decomposition naturally defines $A$ as a direct product of cyclic rings,
which makes $A$ a commutative ring with unity and allows us to define
the group ring $A[B]$.
The mapping $\tau\colon (A^B,+) \to (A[B],+)$ defined by
$$
f\ \stackrel{\tau}{\mapsto}\ \sum_{b\in\supp(f)} f(b)\cdot b,
$$
is an abelian group isomorphism.
Therefore, we can view $A^B$ as a group ring with 
the ring of scalars $A$ and the basis $B$. Observe that, for a function $f\in A^B$, if $\supp(f) = \{b_1,\ldots, b_k\}$ and $f(b_i)=a_i$, then $\tau(f)=a_1\cdot b_1+\dots+a_k\cdot b_k$.

\begin{lemma}
$
\tau(f)=a_1\cdot b_1+\dots+a_k\cdot b_k
\ \ \ \Leftrightarrow\ \ \ 
\tau(f^\delta)=a_1\cdot \delta^{-1}b_1+\dots+a_k\cdot \delta^{-1}b_k
$
for every $f\in A^B$ and $\delta\in B$.
\end{lemma}

\begin{proof}
For every $x\in B$ and $a\in A$, we have
$
f(x)=a
\ \ \Leftrightarrow\ \ 
f^\delta(\delta^{-1}x) = a.
$ Thus the statement holds.
\end{proof}

Therefore, multiplication on $A[B]$ induces multiplication on $A^B$
\[
\rb{\sum_{i=1}^k a_i^{b_i}} 
\cdot 
\rb{\sum_{j=1}^l c_j^{d_j}} 
=
\sum_{i=1}^k\sum_{j=1}^l a_i^{b_i} c_j^{d_j} 
=
\sum_{i=1}^k\sum_{j=1}^l (a_ic_j)^{b_i+d_j}.
\]
Denote by $1$ the unity of $A$. Then $1^0$ is the unity of $A^B$.
Observe that, for any $f\in A^B$ and $b\in B$, the following identities hold:
\begin{enumerate}
\item[1)] $f^b = f\cdot 1^b$,
\item[2)]
$f^0 =f=f\cdot 1^0$,
\item[3)]
$f-f^b = f(1^0-1^b)$.
\end{enumerate}

\subsection{Some properties of $A^B$}
\label{se:A-B-properties}

Let us introduce some notation.
For a subgroup $N\unlhd B$, let $\varphi\colon B\to B/N$ be the
canonical epimorphism defined by $b\stackrel{\varphi}{\mapsto}b+N$
and $\varphi^\ast\colon A^B\to A^{B/N}$ the induced epimorphism
defined for $f\in A^B$ and $x\in B/N$ by
$$
\varphi^\ast(f)(x) = \sum_{\varphi(y)=x} f(y).
$$
For $f,g\in A^B$, $f=_{A^{B/N}}g$ means that $\varphi^\ast(f)=\varphi^\ast(g)$.

For $\ovb=(b_1,\ldots,b_k)\in B^k$, define 
a subgroup $N=\gp{b_1,\ldots,b_k}\le B$, 
a linear function $\lambda_{\ovb}\colon (A^B)^k\to A^B$ by
$$
(f_1,\ldots,f_k)
\ \ \stackrel{\lambda_{\ovb}}{\mapsto}\ \ 
f_1(1^0-1^{b_1})+\cdots+f_k(1^0-1^{b_k}),
$$ 
and denote the epimorphism $\varphi^\ast\colon A^B\to A^{B/N}$ 
defined above by $\pi_\ovb$.
It is easy to check that both functions are homomorphisms
and that $\im(\lambda_{\ovb})\subseteq\ker(\pi_{\ovb})$
because
\begin{align*}
g\in\im(\lambda_{\ovb})
&\ \ \Rightarrow\ \  
g=f_1(1^0-1^{b_1})+\cdots+f_k(1^0-1^{b_k})
\mbox{ for some }f_1,\ldots,f_k\in A^B\\
&\ \ \Rightarrow\ \  
\pi_{\ovb}(g)=
\pi_{\ovb}(f_1)\pi_{\ovb}(1^0-1^{b_1}) + \cdots + \pi_{\ovb}(f_k)\pi_{\ovb}(1^0-1^{b_k})=0.
\end{align*}

\begin{lemma}\label{le:lambda-image}
$a^{x}-a^{x+\delta}\in \im(\lambda_\ovb)$
for every $a\in A$, $x\in B$, and $\delta\in N$.
\end{lemma}

\begin{proof}
Since $\delta\in N$, then $\delta=\alpha_1 b_1+\dots+\alpha_k b_k$
for some $\alpha_1,\dots,\alpha_k\in\MZ$.
We proceed using induction on $\alpha=|\alpha_1|+\dots+|\alpha_k|$.
If $\alpha=0$, then $\delta=0$ and there is nothing to prove.
Otherwise, without loss of generality, assume that $\alpha_1>0$. Let $\delta'=(\alpha_1-1) b_1+\dots+\alpha_k b_k=\delta-b_1$. Let $\alpha'$ be the corresponding $\alpha$-value for $\delta'$. Then we have
\begin{align*}
a^{x}-a^{x+\delta}
&=(a^{x}-a^{x+\delta'}) +(a^{x+\delta'} -a^{x+\delta})\\
&=
(a^{x}-a^{x+\delta'})+
a^{x+\delta'}(1^0-1^{b_1}),
\end{align*}
where $a^{x+\delta'}(1^0-1^{b_1})\in\im(\lambda_\ovb)$
and $a^{x}-a^{x+\delta'}$ is an element of the same type with $\alpha'<\alpha$.
Hence, by the induction assumption, $a^{x}-a^{x+\delta'} \in\im(\lambda_\ovb)$
and $a^{x}-a^{x+\delta}\in \im(\lambda)$.
\end{proof}

\begin{proposition}\label{cor:sequencemulti}
$\im(\lambda_{\ovb})\ =\ \ker(\pi_{\ovb})$.
\end{proposition} 

\begin{proof}
We have seen that $\im(\lambda_{\ovb})\subseteq\ker(\pi_{\ovb})$. Let $g\in A^B$ be a nontrivial function such that $g\in\ker(\pi_\ovb)$. We now show that $\ker(\pi_{\ovb})\subseteq\im(\lambda_{\ovb})$ using induction on $|\supp(g)|$. For the base case, we use $|\supp(g)|=2$; the statement in this case holds by Lemma \ref{le:lambda-image}. Assume that the statement holds up to $|\supp(g)|=n$, and consider the case where $|\supp(g)|=n+1$. Since $g\in\ker(\pi_\ovb)$, then at least two elements from $\supp(g)$ collapse modulo $N$, i.e., for some $a,c\in A$, $x\in B$ and $\delta\in N$ we have
$$
g = a^x+c^{x+\delta}+g' = 
(a^x-a^{x+\delta}) + (c+a)^{x+\delta}+g'.
$$
By Lemma \ref{le:lambda-image},
$a^x-a^{x+\delta} \in \im(\lambda_\ovb) \subseteq \ker(\pi_{\ovb})$. Hence, $(c+a)^{x+\delta}+g' \in \ker(\pi_{\ovb})$. Since $|\supp((c+a)^{x+\delta}+g')|<|\supp(g)|=n+1$, the statement holds by mathematical induction.
\end{proof}

\begin{proposition}\label{pr:transform-components}
Consider $W=\sum_{i=1}^k f_i(1^0-1^{b_i}) + W'$, 
where $b_1,\dots,b_k\in B$ are constants,
$f_1,\dots,f_k,\ldots\in A^B$ are unknowns, and $W'$ is independent of $f_1,\dots,f_k$. Let $N=\gp{b_1,\dots,b_k}$. The following are equivalent:
\begin{itemize}
\item[(a)]
$W=0$ has a solution in $A^B$,
\item[(b)] 
$W=0$ has a solution in $A^{B/N}$,
\item[(c)]
$W'=0$ has a solution in $A^{B/N}$.
\end{itemize}
\end{proposition}

\begin{proof}
Equivalence of (b) and (c) is obvious. Below we prove 
equivalence of (a) and (b).

``(a)$\Rightarrow$(b)''
Use homomorphism $\pi_{\ovb}\colon A^B\to A^{B/N}$.

``(a)$\Leftarrow$(b)''
Suppose that $f_1',\ldots,f_k',\ldots$ satisfy $W=0$ in $A^{B/N}$.
Choose any $f_1^\ast,\ldots,f_k^\ast,\ldots\in A^B$ satisfying $f_i'=\pi_\ovb(f_i^\ast)$. 
Then $\pi_\ovb(W(f_1^\ast,\ldots,f_k^\ast,\ldots))=0$ and we have
$$
\underbrace{\sum f_i^\ast(1^0-1^{b_i})}
_{\mbox{\scalebox{0.8}{belongs to}} \im(\lambda_\ovb)}
+W'(\cdots)=
W(f_1^\ast,\ldots,f_k^\ast,\ldots)\in\ker(\pi_\ovb)\ \stackrel{2.3}{=}\ \im(\lambda_\ovb).
$$
Hence, $W'(\cdots) \in \im(\lambda_\ovb)$
and, since $W'$ is independent of $f_1,\dots,f_k$,
there exist $f_1,\ldots,f_k\in A^B$ satisfying
$\sum_{i=1}^k f_i(1^0-1^{b_i})=-W'(\cdots)$. 
Thus, $W=0$ has a solution in $A^B$.
\end{proof}

\subsection{Several useful formulae}
We frequently make use of the following formulae throughout the paper (where $x=(\delta_x,f_x)$, $y=(\delta_y,f_y)$, $z=(\delta_z,f_z)$, and $c=(\delta_c,f_c)$):
\begin{align}
    (\delta,f)^{-1} & = (-\delta,-f^{-\delta}), \\
    z^{-1}cz & = (\delta_c,f_z(1^0-1^{\delta_c})+f_c^{\delta_z}), \\
    [x,y] = x^{-1}y^{-1}xy & = (0,f_y(1^0-1^{\delta_x})-f_x(1^0-1^{\delta_y})).
\end{align}

\subsection{Cayley graph}\label{se:cayley}

Let $G$ be a group with a finite generating set $X$.
By $\Gamma_G$ we denote the Cayley graph for $G$ with respect to $X$.
Note that $G$ is a discrete metric space in which, for any $g,h\in G$,
the \emph{distance} from $g$ to $h$, $d(g,h)$, 
is the distance from $g$ to $h$ in $\Gamma_G$.
Let $S\subset G$. Then we define the \emph{diameter} of $S$, $\diam(S)$, as
\[
\diam(S) = \sup\Set{d\rb{g,h}}{g,h\in S}.
\]
Similarly, we define the diameter of a function 
$f=a_1^{-g_1}+\cdots+a_k^{-g_k}\in A^G$ as
\[
\diam(f) = \diam\rb{\left\{g_1,\ldots,g_k\right\}}.
\] 
The \emph{ball} of radius $r$ in the Cayley graph of $G$ is the set 
$$
\mathcal{B}_r=\Set{g\in G}{d(1,g)\le r}.
$$ 
Define the \emph{growth function} $\gr_X(n)\colon\MN\to\MN$ of $G$ 
as $\gr_X(n) = |\mathcal{B}_n|$.
The \emph{geodesic length} $|g|$ of $g\in G$ is the minimal $k$ such that $g\in \mathcal{B}_k$.
It is a well-known fact that the growth function of an abelian group is bounded by a polynomial.

\subsection{The rank of an abelian group}
\label{se:rank}

Let $B$ be a finitely generated abelian group.
In this paper by the \emph{rank} of a finitely generated abelian group $B$ we mean the minimum number of generators for $B$.
Suppose that $B$ is given by the following data:
\begin{itemize}
\item 
a set of generators $b_1,\ldots,b_n$ for $B$;
\item 
a set of relations $\ovr_1,\ldots,\ovr_m \in \MZ^n$ for $B$.
\end{itemize}
The problem of finding $\rank(B)$ can be 
solved in polynomial time by
constructing a matrix of the relations 
$\ovr_1,\ldots,\ovr_m$
and computing its Smith normal form
(see \cite{Storjohann:1996}), which gives a presentation for $B$ as a direct product of cyclic groups
$$
B\simeq 
\MZ_{\alpha_1}\times 
\dots
\times 
\MZ_{\alpha_y}\times 
\MZ^{k-y},
$$
where $\alpha_1\mid\alpha_2\mid \dots \mid\alpha_y$. 
The rank of $B$ is the number of nontrivial factors in the product above, i.e.,
$\rank(B) = k-\max\{i\mid \alpha_i=1 \}$.

Let $S\le B$, where $B$ is given as above and generators for $S$, viewed as the corresponding linear combinations of $b_1,\ldots,b_n$, are given by $\ovs_1,\ldots,\ovs_k \in\MZ^n$. We can then efficiently solve a similar problem of finding the rank of $S$. For that it is sufficient to find the space of relations for $S$. Consider a subgroup of $\MZ^{n+k}$ generated by concatenated vectors
$$
E=\gp{
\ovr_1 \ovo, \dots, \ovr_m \ovo,
\ovs_1 \ove_1, \dots, \ovs_k \ove_k},
$$
where $\{\ove_1,\dots,\ove_k\}$ is the standard basis for $\MZ^k$. 
Notice that 
$$
(t_1,\ldots,t_k) \mbox{ is a relation for } \ovs_1,\ldots,\ovs_k
\ \ \Leftrightarrow\ \ 
(\underbrace{0,\dots,0}_{n},t_1,\ldots,t_k)\in E.
$$
The set of element $(0,\dots,0,t_1,\ldots,t_k)\in E$
can be captured by solving the following linear vector-equation:
$$
\alpha_1\ovr_1 + \dots + \alpha_m\ovr_m+
\beta_1\ovs_1 + \dots + \beta_k\ovs_k = \ovo
$$
for $\alpha_1,\dots,\alpha_m,\beta_,\dots,\beta_k\in\MZ$.
We can find a basis for the space of its solutions 
and then truncate the basis vectors by
deleting the first $m$ components (the $\alpha$-part).
The result generates the space of relations for $S$.
Once relations for $\ovs_1,\ldots,\ovs_k$ are found, we construct 
a matrix of those relations and find $\rank(S)$ 
the same way as for $B$.
This can be done in polynomial time.

\subsection{Lattices}
We now present several definitions and lemmas related to integer lattices.
Recall that an integer lattice $\CL$ is a subgroup of $\MZ^n$.
Algebraically $\CL\simeq \MZ^k$, where $k\le n$ is called the \emph{rank} of $\CL$.
Associated with a given lattice $\CL$ is a sequence of positive real values
$\lambda_1\le\cdots\le\lambda_n$, 
called the \textit{successive minima} of $\CL$, defined as
$$
\lambda_i(\CL)=\min\Set{r}{\dim(\spn(\mathcal{B}_r\cap\CL))\geq i}.
$$
Equivalently, $\lambda_i(\CL)$ is the radius of the smallest ball in $\CL$ containing $i$ linearly independent lattice vectors. The values of $\lambda_i(\CL)$ are independent of the choice of basis for $\CL$.  

The next lemma follows immediately from the definition of successive minima.

\begin{lemma}\label{le:lambda_bound}
Suppose that $N=\gp{s_1,\ldots,s_n}\le \MZ^m$, where $\|s_i\|\le k$ and $\rank(N)=y$. Then $\lambda_i(N)\le k$ for every $i=1,\ldots,y$.
\end{lemma}

From the Gram--Schmidt orthogonalization process, define the orthogonalized vectors $\bv_i^\ast$ and the $i$th orthogonal projection $\pi_i$ by
$$
\bv_i^\ast=\bv_i-\sum_{j<i}\frac{\gp{\bv_i,\bv_j^\ast}}{\|\bv_j^\ast\|^2}\bv_j^\ast, \ \ \ \ \pi_i(\bv)=\sum_{j=i}^n \frac{\gp{\bv,\bv_j^\ast}}{\|\bv_j^\ast\|^2}\bv_j^\ast.
$$
Clearly, the Gram--Schmidt orthogonal vectors can be represented as $\{\bv_1^\ast,\ldots,\bv_m^\ast\}$, where $\bv_i^\ast=\pi_i(\bv_i)$.

\begin{definition}[{{\cite[Definition 2.4]{Micciancio-Goldwasser:2002}}}]
A basis $\{\bv_1,\ldots,\bv_m\}$ is \textit{LLL-reduced} with factor $\delta$ if the following two conditions are satisfied:
\begin{enumerate}
    \item The basis $\{\bv_1,\ldots,\bv_m\}$ is size-reduced, i.e. if
    $
    \left|\frac{\gp{\bv_i,\bv_j^\ast}}{\|\bv_j^\ast\|^2}\right|\le\frac{1}{2}
    $
    for all $i>j$.
    \item For any consecutive $\bv_i,\bv_{i+1}$ in the basis,
    $
    \delta\|\bv_i^\ast\|^2\le\|\pi_i(\bv_{i+1})\|^2.
    $
\end{enumerate}
\end{definition}

An algorithm to compute LLL reduced bases can be found in 
\cite[Section 2.2]{Micciancio-Goldwasser:2002}.
For our purposes, we use $\delta=\frac{3}{4}$.

\begin{theorem}[{{\cite[Chapter 2, Theorem 9]{Nguyen-Vallee:2010}}}]\label{th:nguyen}
Suppose that $\{\bv_1,\ldots,\bv_y\}$ is an LLL-reduced basis with factor $\delta=\frac{3}{4}$ of a lattice $\CL$ in $\MR^n$. Then
$
\|\bv_i\|\le \sqrt{2^{y-1}}\lambda_i(\CL)
$
for all $i=1,\ldots,y$.
\end{theorem}

This leads to the following two natural corollaries.

\begin{corollary}\label{co:cor_to_nguyen}
Suppose that $N=\gp{s_1,\ldots,s_n}\le \MZ^m$, where $\|s_i\|\le k$ and $\rank(N)=y$. Then $N$ has an LLL-reduced basis $\{t_1,\ldots,t_y\}$ with factor $\delta=\frac{3}{4}$, satisfying
\begin{equation}\label{eq:norm_bound_k}
   \|t_i\|\le \sqrt{2^{y-1}}k
\end{equation}
for all $i=1,\ldots,y$.
\end{corollary}

\begin{proof}
Let $\{t_1',\ldots,t_y'\}$ be a basis for $N$. Clearly, $N$ is a sublattice of $\MZ^m$, so we can run the LLL lattice basis reduction algorithm on $\{t_1',\ldots,t_y'\}$ to obtain an LLL-reduced basis $\{t_1,\ldots,t_y\}$ with factor $\delta=\frac{3}{4}$. It then follows that, for every $i\in\{1,\ldots,y\}$,
$$
\|t_i\|\stackrel{\ref{th:nguyen}}{\le} \sqrt{2^{y-1}}\lambda_i(\CL) \stackrel{\ref{le:lambda_bound}}{\le} \sqrt{2^{y-1}}k.
$$
\end{proof}

\begin{corollary}\label{cor:bounded_gen_set}
Suppose that $B$ is a rank $m$ abelian group, and let 
$N=\gp{s_1,\ldots,s_n}\le B$, where $\|s_i\|\le k$
and $\rank(N)=y$. Then $N$ has a generating set
$\{t_1,\dots,t_y\}$ satisfying $\eqref{eq:norm_bound_k}$.
\end{corollary}

\begin{proof}
By the fundamental theorem of finitely generated abelian groups, $B\simeq \MZ_{\alpha_1}\times\cdots\times\MZ_{\alpha_m}$, where trailing $\alpha_i$'s can be infinite. Thus every element of $B$ is defined by an $m$-tuple of integers. Let $\gamma\colon\MZ^m\to B$ and  $\gamma_i\colon\MZ\to\MZ_{\alpha_i}$ be canonical epimorphisms. For an element $s_i\in B$, define a preimage $\gamma^{-1}(s_i)\in\MZ^m$ as the embedding of $s_i$ in $\MZ^m$. Since $N\le B$, a preimage of $N$ in $\MZ^m$ is $N^\ast=\gp{\gamma^{-1}(s_1),\ldots,\gamma^{-1}(s_n)}$. Since $\|s_i\|=\|\gamma^{-1}(s_i)\|$,
$$
\|s_i\|\le k \ \Rightarrow \   \|\gamma^{-1}(s_i)\| \le k,
$$
and since $\|\gamma^{-1}(s_i)\| \le k$ and $\rank(N^\ast)=y$, by Corollary \ref{co:cor_to_nguyen}, $N$ has an LLL-reduced basis $\{t_1^\ast,\ldots,t_y^\ast\}$ with factor $\delta=\frac{3}{4}$ satisfying \eqref{eq:norm_bound_k}.
Mapping this bounded basis to $B$, we obtain $\{t_1,\ldots,t_y\}$, where $t_i=\gamma(t_i^\ast)$ as illustrated below for $t^\ast=(u_1^\ast,\ldots,u_m^\ast)$:
$$
\begin{array}{rccc}
t^\ast= & \big(u_1^\ast\in\MZ, & \cdots & ,u_m^\ast\in\MZ\big) \\[10pt]
& \downarrow\gamma_1 &  & \downarrow\gamma_m \\[10pt]
t = & \big(u_1\in\MZ_{\alpha_1}, & \cdots & ,u_m\in\MZ_{\alpha_m}\big).
\end{array}
$$
Clearly, components of $t_i^\ast$ can only decrease in norm, i.e. $\|u_i\|\le\|u_i^\ast\|$, and it follows that
$$
\|t_i\|\le\|t_i^\ast\|\le \sqrt{2^{y-1}}k,
$$
and we are done.
\end{proof}

\subsection{The $3$-partition problem}
For a multiset $S=\{s_1,\ldots,s_{3k}\}$ of $3k$ integers, let 
$$
T_S = \frac{1}{k}\sum_{i=1}^{3k}s_i.
$$

\begin{algproblem}
 \problemtitle{\textsc{$3$-partition problem} $(\TPART)$}
  \probleminput{Set $S=\{s_1,\ldots,s_{3k}\}$ with $T_S/4 < t_i < T_S/2$.}
  \problemquestion{Is there a partition of $S$ into $k$ triples, each of which sums to $T_S$?
  }
\end{algproblem}

This problem is known to be \emph{strongly} $\NP$-complete, which means that it remains $\NP$-complete even when the integers in $S$ are bounded above by a polynomial in $m=3k$. In other words, the problem remains $\NP$-complete even when the input is represented in unary.\footnote{A thorough treatment of this problem may be found in \cite{Garey-Johnson:1979}.}


\subsection{Zero-one equation problem}

A vector $\overline{v} \in \MZ^n$ is called a \emph{zero-one} vector
if each entry in $\overline{v}$ is either $0$ or $1$.
Similarly, a square matrix  $M\in \Mat(n,\MZ)$ is called a \emph{zero-one}
matrix if each entry in $M$ is either $0$ or $1$.
Denote by $1_n$ the vector $(1,\ldots,1) \in \MZ^n$.
The following problem is $\NP$-complete  (see \cite{Dasgupta-Papadimitriou-Vazirani:2006}).

\begin{algproblem}
 \problemtitle{\textsc{Zero-one equation problem} $(\ZOE)$}
  \probleminput{Zero-one matrix $M \in \Mat(n,\MZ)$.}
  \problemquestion{Is there a zero-one vector $\ovx \in \MZ^n$ satisfying $M\cdot \ovx = 1_n$?
  }
\end{algproblem}


\section{Quotient sum problem: definition}
\label{se:QSP}

\subsection{Notation and conventions}
We first introduce some convenient notation and conventions for describing certain sequences of variables and constants.

\begin{itemize}
    \item A sequence of functions $f_1,\ldots,f_m\in A^B$ is denoted by $\ovf$. A sequence of functions $f_{x_1},\ldots,f_{x_g}\in A^B$ is denoted by $\ovf_x$. Unless otherwise specified, we always assume that the number of elements in a sequence $\ovf$ or $\ovf_x$ is clear from the context.
    \item When we say that a sequence of functions $\ovf$ is in $A^B$, denoted by $\ovf\in A^B$, we mean that $f_i\in A^B$ for each component $f_i$. Similarly, when we say that a sequence of elements $\ovdelta\in B$, we mean that $\delta_i\in B$ for each component $\delta_i$.
    \item For a sequence $\ovf=(f_1,\ldots,f_m)$ of functions from $A^B$ and a sequence $\ovdelta=(\delta_1,\ldots,\delta_m)$ of elements from $B$ we define the sequence
    $$
    \ovfdelta=\bigl(f_1^{\delta_1},\ldots,f_m^{\delta_m}\bigr).
    $$
    \item Let $g\colon A^B\rightarrow Z$, where $Z$ is any arbitrary set. Then if $\ovf=(f_1,\ldots,f_m)$ is a sequence of functions, $(g(f_1),\ldots,g(f_m))$ is denoted by $g(\ovf)$.
    \item For sequences $\ovdelta=(\delta_1,\ldots,\delta_m)$ and $\ovdelta'=(\delta_1',\ldots,\delta_m')$ we write $\ovdelta\equiv \ovdelta' \bmod N$ if $\delta_i+N=\delta_i'+N$ for each $i=1,\ldots,m$.
    \item When we say that $\ovdelta_x,\ovdelta_y,\ovdelta_z$ satisfy some condition, we mean that their elements satisfy that condition.
\end{itemize}

For an arbitrary group $G$, we use $=_G$ to denote ``is equal to in $G$'' or ``holds in $G$''.

\subsection{The quotient sum problem}

Here we define a decision problem to which we subsequently 
reduce the Diophantine problem for orientable quadratic equations.
\par\addvspace{.5\baselineskip}

\noindent\begin{minipage}{\textwidth}
  \par\addvspace{.5\baselineskip}
  \noindent
  \begin{tabularx}{\textwidth}{@{\hspace{0em}} l X c}
    \hline\hline
    \end{tabularx}\vspace*{-1.5\baselineskip}
\begin{algproblemnl}
 \problemtitle{\textsc{Quotient sum problem} $(\QSP)$}
  \probleminput{
  Tuple $(A,B,\ovf,h)$, where $A$ and $B$ are finitely generated abelian groups, $\ovf\in A^B$, and $h\in\MN$.}
  \problemquestion{Is there a sequence $\ovdelta\in B$ and a subgroup $N\le B$ such that
  }
\end{algproblemnl}
\vspace*{-\baselineskip}
\begin{equation}\label{eq:qsp_shift}
    \sum_{i=1}^m f_i^{\delta_i}=_{A^{B/N}}0 \
    \mbox{ and } \ \rank(N)\le h?
\end{equation}
\noindent
  \begin{tabularx}{\textwidth}{@{\hspace{0em}} l X c}
    \hline\hline
    \end{tabularx}
\end{minipage}

\par\addvspace{.5\baselineskip}
Some details of data representation are in order. 
A finitely generated abelian group $A$ is given by a set of generators
(symbols) $a_1,\ldots,a_n$ and a finite set of relations $R_A\subseteq \MZ^n$.
Elements of $A$ are defined by vectors $(\alpha_1,\dots,\alpha_n)\in\MZ^n$ 
using the formula
$$a=\alpha_1a_1+\dots+\alpha_n a_n,$$
which is viewed as a composition
of generators (and their inverses), 
i.e. $\alpha_i$'s are written in unary.
The size of a given representation $(\alpha_1,\dots,\alpha_n)$ of $a$ is
$$
\size((\alpha_1,\dots,\alpha_n)) = |\alpha_1|+\dots+|\alpha_n|.
$$
Presentations of elements of $B$ and their sizes are defined analogously. 

A function $f\in A^B$ with $\supp(f)=\{b_1,\ldots,b_m\}$ and $f(b_i)=a_i$ 
is defined by a finite set of pairs $\{(a_1,b_1),\ldots,(a_m,b_m)\}$. 
The size of this representation of $f$ is
$$
\size(f) = \sum_{i=1}^m (\size(a_i) + \size(b_i)).
$$

Given an instance $I=(A,B,\ovf,h)$ of $\QSP$, we can define $\size(I)$ as
$$
\size(I)=\size(A)+\size(B)+\size(\ovf)+h,
$$
where $\size(A)$ (or $\size(B)$) is the size of the presentation of $A$ (or $B$) and $\size(\ovf)=\sum \size(f_i)$.

\subsection{$\mathsf{QSP}$ as a function}
By definition, $\QSP$ involves multiple parameters, namely $A$, $B$, $\ovf$, and $h$. To study the problem in more detail and explore the interactions
between these parameters, it will be convenient to define $\QSP$ as a function
$$
(A,B,\ovf,h)\ \stackrel{\QSP}{\longmapsto}\ \{0,1\}.
$$
That allows us to use the following notation for versions of $\QSP$:
\begin{itemize}
\item 
$\QSP(\cdot,\cdot,\cdot,\cdot)$ is the most general problem
in which no parameters are fixed.
\item 
$\QSP(A,B,\cdot,\cdot)$ is the 
problem with fixed groups $A$ and $B$.
\item 
$\QSP(A,B,\cdot,0)$
is the problem with fixed $A$ and $B$ and $h=0$.
\item 
$\QSP(A,B,\cdot_{|\ovf|\le k},\cdot)$
is the problem with fixed $A$ and $B$ and with the length $|\ovf|$
(the number of components) of $\ovf$ bounded by a fixed $k$.
\end{itemize}
These problems are important for our study of orientable
equations in wreath products.

\begin{remark}
$\QSP(\mathbbm{1},\cdot,\cdot,\cdot)$ (where $\mathbbm{1}$ is the trivial group) is decidable in constant time because
every instance of $\QSP(\mathbbm{1},\cdot,\cdot,\cdot)$ is positive.
Therefore, from now on we assume that $A\not\simeq \mathbbm{1}$.
\end{remark}

\subsection{$\mathsf{QSP}$: motivation}

Consider an orientable quadratic equation $\CE$ of type
\eqref{eq:orientable} of genus $g$ over $A\wr B$, 
where $x_i$, $y_i$, and $z_i$ are variables and $c_i$ are constants. 
Below we show that the decidability of $\CE$ can be reduced
to an instance of $\QSP$.

\begin{lemma}\label{le:orientable-to-A_B-simpl}
Let $c_i=(\delta_{c_i},f_{c_i})$. Then $x_i=(\delta_{x_i},f_{x_i})$, $y_i=(\delta_{y_i},f_{y_i})$, and $z_i=(\delta_{z_i},f_{z_i})$ satisfy \eqref{eq:orientable} if and only if $\sum_{i=1}^m\delta_{c_i}=0$ and $f_{x_i}$, $f_{y_i}$, $f_{z_i}'=f_{z_i}1^{\sum_{j=i+1}^m \delta_{c_j}}$, $\delta_{x_i}$, $\delta_{y_i}$, $\delta_{z_i}'=\delta_{z_i}+\sum_{j=i+1}^m\delta_{c_j}$ satisfy
    \begin{equation}\label{eq:orientable-to-A_B-simpl}
    \sum_{i=1}^g \biggl[f_{y_i}\rb{1^0-1^{\delta_{x_i}}}-f_{x_i}\rb{1^0-1^{\delta_{y_i}}}\biggr] +
    \sum_{i=1}^m \biggl[f_{z_i}'\rb{1^0-1^{\delta_{c_i}}}+f_{c_i}^{\delta_{z_i}'}\biggr]=0.
    \end{equation}
\end{lemma}

\begin{proof}
Straightforward computation.
\end{proof}

\begin{proposition}\label{pr:DP-QSP-reduction}
Let $\CE$ be an equation of the form \eqref{eq:orientable}
satisfying $\sum_{i=1}^m\delta_{c_i}=0$.
Then $\CE$ has a solution if and only if $(A,B/\gp{\ovdelta_c},\ovf_c,2g)$
is a positive instance of $\QSP$.
\end{proposition}

\begin{proof}
Indeed, 
\begin{align*}
\mbox{$\CE$ has a solution in $A\wr B$}
&\ \ \stackrel{\ref{le:orientable-to-A_B-simpl}}{\Leftrightarrow}\ \ 
\mbox{\eqref{eq:orientable-to-A_B-simpl} has a solution in $A^B$}\\
&\ \ \stackrel{\ref{pr:transform-components}}{\Leftrightarrow}\ \ 
\sum_{i=1}^g \biggl[f_{y_i}\rb{1^0-1^{\delta_{x_i}}}-f_{x_i}\rb{1^0-1^{\delta_{y_i}}}\biggr] +
\sum_{i=1}^m f_{c_i}^{\delta_{z_i}'}=_{A^{B/\langle\ovdelta_c\rangle}}0\\
&\ \ \stackrel{\ref{pr:transform-components}}{\Leftrightarrow}\ \ 
\sum_{i=1}^m f_{c_i}^{\delta_{z_i}'}=_{A^{(B/\langle\ovdelta_c\rangle)/\langle\ovdelta_x,\ovdelta_y\rangle}}0 \\
&
\ \ \Leftrightarrow\ \ 
(A,B/\gp{\ovdelta_c},\ovf_c,2g)
\mbox{ is a positive instance of $\QSP$.}
\vspace*{-\dimexpr\baselineskip + \topsep}\qedhere
\end{align*}
\end{proof}

\section{Quotient sum problem: computational properties}
\label{se:QSP-properties}

In this section we study computational properties of 
several variations of $\QSP$ that correspond to 
the Diophantine problem for some natural classes of equations
over wreath products $A\wr B$
(discussed in Section \ref{se:orientable}).

\begin{lemma}\label{le:trivial-sum}
If $I=(A,B,\ovf,h)$, where $\ovf=(f_1,\dots,f_m)$,
is a positive instance of $\QSP$, then 
$\sum_{i=1}^m \sum_{x\in\supp(f_{i})} f_{i}(x)= 0$.
\end{lemma}

\begin{proof}\belowdisplayskip=-12pt
\begin{align*}
\mbox{$I$ is a positive instance of $\QSP$}
&\ \Rightarrow\ 
\sum_{i=1}^m f_i^{\delta_i}=_{A^{B/N}}0
&&\mbox{(for some } N\le B \mbox{ and } \ovdelta\in B\mbox{)}\\
&\ \Rightarrow\ 
\sum_{i=1}^m f_i^{\delta_i}=_{A^{B/B}}0 \\
&\ \Rightarrow\ 
\sum_{i=1}^m \sum_{x\in\supp(f_{i})} f_{i}(x)= 0.
\end{align*}
\end{proof}

\subsection{Complexity lower bound}

Let $a\in A$ and $b\in B$.
For an instance $T=\{t_1,\ldots,t_{3k}\}$ of $\TPART$,
let $L = \frac{1}{k}\sum t_i$ be the target sum for subsets, where $L/4<t_i<L/2$ for each $i$.
For $y\in \MN$ define elements
$$
c_y= \prod_{i=0}^{y-1} a^{-i\cdot b}, \ \ \ \  
c=\prod_{i=0}^{k-1} 
c_L^{-(L+1)i\cdot b}.
$$
Let $\ovc=(c_{t_1},\ldots,c_{t_{3k}},c^{-1})$, where, slightly abusing notation, $c_{t_i}$ are interpreted as the function components of the corresponding elements and $c$ is interpreted as the function component of $c$. The proof of the following proposition is similar in structure to the proof of \cite[Proposition 3.1]{ushakov2023quadratic}.

\begin{figure}[!h]
\centering
\scalebox{0.85}{
\begin{tikzpicture}
	\draw[<-,ultra thick] (-1.5,0)--(1.5,0);
        \draw[->,ultra thick] (2.5,0)--(4.5,0);
	\draw[-, thick] (-1,-0.2)--(-1,0.2);
	\draw[-, thick] (0,-0.2)--(0,0.2);
	\draw[-, thick] (4,-0.2)--(4,0.2);
	\draw[-, thick] (3,-0.2)--(3,0.2);
	\draw[-, thick] (1,-0.2)--(1,0.2);
 
	\node[draw=none] at (0,-0.5) {$0$};
        \node[draw=none] at (1,-0.5) {$1$};
        \node[draw=none] at (2,0) {$\cdots$};
        \node[draw=none] at (3,-0.5) {$y-1$};
        \node[draw=none] at (4,-0.5) {$y$};
        \node[draw=none] at (-1,-0.5) {$-1$};
	
	\filldraw[draw=black,fill=black] (1,0) circle (3pt);
        \filldraw[draw=black,fill=white] (-1,0) circle (3pt);
	\filldraw[draw=black,fill=black] (0,0) circle (3pt);
	\filldraw[draw=black,fill=black] (3,0) circle (3pt);
	\filldraw[draw=black,fill=white] (4,0) circle (3pt);
        \node[draw=none] at (1.5,0.5) {\footnotesize{$y$ lamps}};

        \draw [-] (-0.4,0) to (-0.4,1);
        \draw [-] (3.4,0) to (3.4,1);
        \draw [-] (-0.4,1) to (3.4,1);
        
	\end{tikzpicture}}
\caption{The function component of the element $c_y$ in $\MZ_2\wr\MZ$.}\label{fig:cylamps}
\end{figure}
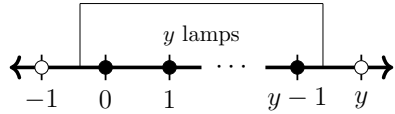

\begin{figure}[!h]
\centering
\scalebox{0.85}{
\begin{tikzpicture}
	\draw[<-,ultra thick] (-1.5,0)--(1.5,0);
        \draw[-,ultra thick] (2.5,0)--(6.5,0);
        \draw[-,ultra thick] (7.5,0)--(9.5,0);
        \draw[-,ultra thick] (10.5,0)--(13.5,0);
        \draw[->,ultra thick] (14.5,0)--(16.5,0);
	\draw[-, thick] (-1,-0.2)--(-1,0.2);
	\draw[-, thick] (0,-0.2)--(0,0.2);
	\draw[-, thick] (5,-0.2)--(5,0.2);
	\draw[-, thick] (4,-0.2)--(4,0.2);
	\draw[-, thick] (3,-0.2)--(3,0.2);
	\draw[-, thick] (1,-0.2)--(1,0.2);
        \draw[-, thick] (6,-0.2)--(6,0.2);
        \draw[-, thick] (8,-0.2)--(8,0.2);
        \draw[-, thick] (9,-0.2)--(9,0.2);
        \draw[-, thick] (11,-0.2)--(11,0.2);
        \draw[-, thick] (12,-0.2)--(12,0.2);
        \draw[-, thick] (13,-0.2)--(13,0.2);
        \draw[-, thick] (16,-0.2)--(16,0.2);

        \node[draw=none, rotate=45, anchor=east] at (-1,-0.25) {\footnotesize$-1$};
	\node[draw=none, rotate=45, anchor=east] at (0,-0.25) {\footnotesize$0$};
        \node[draw=none, rotate=45, anchor=east] at (1,-0.25) {\footnotesize$1$};
        \node[draw=none] at (1.5,0.5) {\footnotesize{$L$ lamps}};
        \node[draw=none] at (2,0) {$\cdots$};
        \node[draw=none, rotate=45, anchor=east] at (3,-0.25) {\footnotesize$L-1$};

        \node[draw=none, rotate=45, anchor=east] at (4,-0.25) {\footnotesize$L$};
        \node[draw=none, rotate=45, anchor=east] at (5,-0.25) {\footnotesize$L+1$};
        \node[draw=none, rotate=45, anchor=east] at (6,-0.25) {\footnotesize$L+2$};
        \node[draw=none] at (6.5,0.5) {\footnotesize{$L$ lamps}};
        \node[draw=none] at (7,0) {$\cdots$};
        \node[draw=none, rotate=45, anchor=east] at (8,-0.25) {\footnotesize$2L$};
        \node[draw=none, rotate=45, anchor=east] at (9,-0.25) {\footnotesize$2L+1$};

        \node[draw=none] at (10,0) {$\cdots$};
        \node[draw=none, rotate=45, anchor=east] at (11,-0.25) {\footnotesize$(L+1)(k-1)-1$};
        \node[draw=none, rotate=45, anchor=east] at (12,-0.25) {\footnotesize$(L+1)(k-1)$};
        \node[draw=none, rotate=45, anchor=east] at (13,-0.25) {\footnotesize$(L+1)(k-1)+1$};
        \node[draw=none] at (14,0) {$\cdots$};
        \node[draw=none] at (13.5,0.5) {\footnotesize{$L$ lamps}};
        \node[draw=none, rotate=45, anchor=east] at (15,-0.25) {\footnotesize$(L+1)(k-1)+(L-1)$};
        \node[draw=none, rotate=45, anchor=east] at (16,-0.25) {\footnotesize$(L+1)(k-1)+L$};

	\filldraw[draw=black,fill=black] (1,0) circle (3pt);
	\filldraw[draw=black,fill=black] (3,0) circle (3pt);
        \filldraw[draw=black,fill=white] (-1,0) circle (3pt);
	\filldraw[draw=black,fill=black] (0,0) circle (3pt);
	\filldraw[draw=black,fill=white] (4,0) circle (3pt);
	\filldraw[draw=black,fill=black] (5,0) circle (3pt);
        \filldraw[draw=black,fill=black] (6,0) circle (3pt);
        \filldraw[draw=black,fill=black] (8,0) circle (3pt);
        \filldraw[draw=black,fill=white] (9,0) circle (3pt);
        \filldraw[draw=black,fill=white] (11,0) circle (3pt);
        \filldraw[draw=black,fill=black] (12,0) circle (3pt);
        \filldraw[draw=black,fill=black] (13,0) circle (3pt);
        \filldraw[draw=black,fill=black] (15,0) circle (3pt);
        \filldraw[draw=black,fill=white] (16,0) circle (3pt);

        \draw [-] (-0.4,0) to (-0.4,1);
        \draw [-] (3.4,0) to (3.4,1);
        \draw [-] (-0.4,1) to (3.4,1);

        \draw [-] (4.6,0) to (4.6,1);
        \draw [-] (8.4,0) to (8.4,1);
        \draw [-] (4.6,1) to (8.4,1);

        \draw [-] (11.6,0) to (11.6,1);
        \draw [-] (15.4,0) to (15.4,1);
        \draw [-] (11.6,1) to (15.4,1);
	\end{tikzpicture}
}
\caption{
The function component for the element $c$ in $\MZ_2\wr\MZ$ consists of $k$ clusters of $L$ lit lamps, each cluster separated by a single unlit lamp.}\label{fig:clamps}
\end{figure}
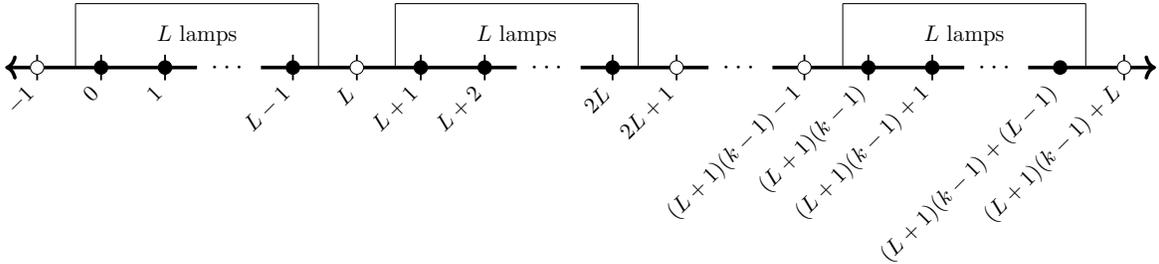

\begin{proposition}\label{pr:qsp_tpart}
Assume $|b|=\infty$ and $a$ is nontrivial. Then $T$ is a positive instance of $\TPART$ if and only if $(A,B,\ovc,0)$ is a positive 
instance of $\QSP$.
\end{proposition}

\begin{proof}
``$\Rightarrow$''
If $T$ is a positive instance of $\TPART$, then, without loss of generality, 
we may assume that
$$
\sum_{i=1}^3 t_{3j+i}=L \ \ \text{for} \ \ j=0,1,\ldots,k-1.
$$
Let $L_i=(L+1)i$ for $i=1,\dots,k-1$.
Then clearly
\begin{align*}
\delta_1&=0 &\delta_2&=-t_1 &\delta_3&=-t_1-t_2\\
\delta_4&=-L_1 &\delta_5&=-L_1-t_4 &\delta_6&=-L_1-t_4-t_5\\
\delta_7&=-L_2 &\delta_8&=-L_2-t_7 &\delta_9&=-L_2-t_7-t_8\\
&\cdots &&\cdots &&\cdots 
\end{align*}
satisfies $\left(\sum_{i=1}^{3k} c_{t_i}^{\delta_i}\right)-c=_{A^{B/\gp{0}}}0$
with $\rank(\gp{0})=0$.
Hence, $(A,B,\ovc,0)$ is a positive instance of $\QSP$.

``$\Leftarrow$''
If $(A,B,\ovc,0)$ is a positive instance of $\QSP$, 
then there is a sequence $\ovdelta=(\delta_1,\ldots,\delta_{3k})$ 
satisfying
$$\rb{\sum_{i=1}^{3k} c_{t_i}^{\delta_i}}-c=_{A^B}0 \ \ \Rightarrow \ \ \sum_{i=1}^{3k} c_{t_i}^{\delta_i}=_{A^{B}}c.$$
By definition of $c_{t_i}$ and $c$ we have
\begin{align*}
\supp(c_{t_i}) &= \{0\cdot b,\ldots,(t_i-1)\cdot b\}\\
\supp(c)&=\{0\cdot b,\ldots,(L-1)\cdot b\}\cup \{(L+1)\cdot b,\ldots,2L\cdot b\}\cup\cdots
\end{align*}
(see Figures \ref{fig:cylamps} and \ref{fig:clamps}) and $\sum |\supp(c_{t_i})|=|\supp(c)|$.
Hence, $\ovdelta$ establishes a matching between elements of
$\cup_i \supp(c_{t_i})$ and $\supp(c)$
by shifting the $3k$ sequences of $a$-lamps
of lengths $t_1,\dots,t_{3k}$ into $k$ groupings of $a$-lamps each of length $L$
(can be viewed as a non-overlapping tiling of $\supp(c)$
with sets $\supp(c_{t_i})$).
Now it is obvious that if the grouping of $L$ $a$-lamps is tiled with
$c_{t_{i_1}},\dots,c_{t_{i_l}}$, then $l=3$ and $t_{i_1}+t_{i_2}+t_{i_3}=L$.
Thus, $T = \left\{t_1,\ldots,t_{3k}\right\}$ is a positive instance of $\TPART$.
\end{proof}

\begin{theorem}\label{th:QSP-infinite-B}
Let $A$ and $B$ be finitely generated abelian groups.
If $A$ is nontrivial and $B$ is infinite,
then $\QSP(A,B,\cdot,0)$ is $\NP$-hard.
\end{theorem}

\begin{proof}
Fix any nontrivial $a\in A$ and any element $b\in B$ of infinite order
(every infinite finitely generated abelian group contains an element of infinite order).
The statement follows immediately from Proposition \ref{pr:qsp_tpart}; the reduction given therein is polynomial-time computable.
\end{proof}

\subsection{Finite group $B$}

In this section we analyze $\QSP$ when the group $B$ is finite.

\begin{theorem}\label{th:QSP-finite-B}
If $|B|<\infty$, then $\QSP(A,B,\cdot,\cdot)$ is in $\P$
for any fixed finitely generated abelian group $A$.
\end{theorem}

\begin{proof}
Fix an instance $I=(A,B,\ovf,h)$ of $\QSP$.
Let $\{a_1,\dots,a_n\}$ be the given generating set for $A$.
Then $X=\Set{a_{i}^{-b}\in A^B}{i=1,\dots,n,\ b\in B}$, where
$$
a_{i}^{-b}(x)=
\begin{cases}
a_i & x=b\\
0 & x\ne b,
\end{cases}
$$
is a finite generating set for the abelian group $A^B$, i.e. $A^B$ is an $n|B|$-generated abelian group. Now consider the following subsets of $A^B$:
\begin{align*}
V_0&=\{0\},&V_{i}&
=V_{i-1}+\Set{f_i^{\delta}}{\delta\in B},& (i&=1,\dots,m),
\end{align*}
where $+$ here is the Minkowski sum. It is straightforward to construct $V_0,\dots,V_m$ starting from $V_0$.
We claim that $V_m$ can be constructed in polynomial time in terms
of $\size(I)$.

As usual, for $g\in A^B$, by $|g|_X$ we denote the length of $g$ in
generators $X$.
Also, notice that, by definition of $X$, $|g^{\delta_i}|_X = |g|_X$
for every $\delta_i\in B$ and $g\in A^B$.
Then
\begin{align*}
f\in V_j
&\ \Rightarrow\ 
f=\sum_{i=1}^j f_i^{\delta_i},
\mbox{ for some } \delta_1,\dots,\delta_j\in B\\
&\ \Rightarrow\ 
|f|_X 
\le \sum_{i=1}^j |f_i^{\delta_i}|_X
\le \sum_{i=1}^j |f_i|_X
\le \sum_{i=1}^m |f_i|_X \le \size(I),
\end{align*}
and hence $f\in \mathcal{B}_{\size(I)}$. 
Recall from Section \ref{se:cayley} that $\gr_X(\size(I))\le p(\size(I))$
for some polynomial $p$. Therefore, $|V_j|\le \gr_X(\size(I))\le p(\size(I))$
for every $j=1,\dots,m$. 
It then follows that $V_m$ can be constructed in polynomial time.

Finally, since $B$ is fixed and finite, 
it has finitely many subgroups $N$ that can be precomputed 
(together with their ranks).
Hence, we can test if condition \eqref{eq:qsp_shift}
holds for some $f\in V_m$ and some $N\le B$ in polynomial time.
\end{proof}

In fact, a stronger statement holds.

\begin{corollary}
$\QSP(\cdot_{\rank(A)\le K},\cdot_{|B|\le m},\cdot,\cdot)$ is in $\P$.
\end{corollary}

\begin{proof}
The proof of Theorem \ref{th:QSP-finite-B} works in this case because
the values of the growth function for $A^B$ can be bounded by the same polynomial
for any groups $A,B$ satisfying $\rank(A)\le K$ and $|B|\le m$.
\end{proof}

At the same time, Theorem \ref{th:QSP-finite-B}
does not generalize to the case when $A$ is allowed to be 
any finitely generated abelian group -- the problem
becomes $\NP$-hard.
Ultimately, the source of hardness for the problem 
is an unbounded rank of $A$.

\begin{theorem}\label{th:B-Z2}
$\QSP(\cdot,\MZ_2,\cdot,0)$ is $\NP$-hard.
\end{theorem}

\begin{proof}
Let $A=\MZ^n$. Consider an instance $M \in \Mat(n,\MZ)$ of $\ZOE$. Denote zero-one column vectors of $M$ by $c_1,\dots,c_n$.
Define functions $f_1,\dots,f_n,f\colon\MZ_2\to A$ by
$$
f_i(x)=
\begin{cases}
c_i & x=0\\
0_n & x=1\\
\end{cases}
\ \ \ \ \mbox{ and }\ \ \ \ 
f(x)=
\begin{cases}
-1_n & x=0\\
1_n-\sum c_i & x=1.\\
\end{cases}
$$
and $\ovf=(f_1,\dots,f_n,f)$.
It is easy to see that $M$ is a positive instance of $\ZOE$ if and only if $\QSP(A,\MZ_2,\ovf,0)=1$.
\end{proof}

Finally, we can make an observation that the choice of the (nontrivial) 
group $A$ is irrelevant for $\NP$-hardness of $\QSP$
in the following sense.

\begin{corollary}\label{co:irrelevant-A}
For any nontrivial finitely generated abelian groups $A_1,A_2$, and $B$
the problems
$\QSP(A_1,B,\cdot,\cdot)$
and
$\QSP(A_2,B,\cdot,\cdot)$
are either both $\NP$-hard or both polynomial-time decidable.
\end{corollary}

\begin{proof}
By Theorems \ref{th:QSP-infinite-B} and
\ref{th:QSP-finite-B}, $\NP$-hardness depends on $B$ only.
\end{proof}

\subsection{Clusters in $A^B$ and $A^{B/N}$}

Let $f_1,\ldots,f_m\in A^B$ and let $I=\{1,\ldots,m\}$ be the set of indices. 
Without loss of generality we may assume that $\supp(f_i)\neq\varnothing$
(otherwise, $f_i$ does not affect the validity of \eqref{eq:qsp_shift}).
Define an equivalence relation $R(\ovf)$ as the transitive closure of the following binary relation on $I$:
$$
J(\ovf)=\Set{(i,j)\in I\times I}{\supp(f_i)\cap\supp(f_j)\neq\varnothing}.
$$
A \emph{cluster} for $\ovf$ is an equivalence class for the relation $R(\ovf)$. 
The set of all clusters for $\ovf$ forms a partition of $I$
and is denoted by $\CC(\ovf)$. 
The graph $G(\ovf)=(I,J(\ovf))$ is called the \emph{cluster-graph} for $\ovf$; clusters are connected components of $G(\ovf)$. For a sequence $\ovf$ of functions $f_1,\ldots,f_m$ and $C\subseteq \{1,\ldots,m\}$
define the \emph{diameter} of $C$ as
$$
\diam(C,\ovf)=
\diam\rb{\bigcup_{i\in C} \supp(f_i)}.
$$

Let $\varphi\colon B\to B/N$ be a canonical epimorphism and
$N=\ker(\varphi)\unlhd B$.
We can define clusters modulo a subgroup $N$ using a parallel construction. Define the equivalence relation $R_\varphi(\ovf)$ as the transitive closure of
$$
J_{\varphi}(\ovf)=\Set{(i,j)\in I\times I}{\varphi(\supp(f_i)) \cap \varphi(\supp(f_j))\ne\varnothing}.
$$
Then a \textit{cluster modulo $N$} for $\ovf$ is an equivalence class for $R_\varphi(\ovf)$, and the set of all clusters modulo $N$ for $\ovf$ is denoted by $\CC_{\varphi}(\ovf)$. The graph $G_\varphi(\ovf)=(I, J_\varphi(\ovf))$ is called the \textit{cluster-graph modulo $N$} for $\ovf$.

\begin{lemma}\label{le:diameterbounded}
$\diam(C,\ovf) \le \sum_{i\in C} \diam(f_i)$
for every $C\in\CC(\ovf)$.
\end{lemma}

\begin{proof}
We proceed with induction on $|C|$. If $|C|=1$, then 
$C=\{i\}$ for some $i\in I$ and
$$
\diam(C,\ovf) = \diam(f_i)
$$
and the statement holds. 
Assume that the statement holds when $|C|=n$ and consider a cluster $C$
of $n+1$ functions. Without loss of generality, $C=\{1,\ldots,n+1\}$
and the cluster-graph $G(\ovf)$ 
has a connected component on vertices $\{1,\ldots,n+1\}$.
That component must have at least one vertex that is not a cut-vertex.
Without loss of generality, $n+1$ is not a cut-vertex. 
Deleting $n+1$ from $G(\ovf)$ creates a connected component $C'=\{1,\ldots,n\}$
and, by the induction assumption, 
$$
\diam(C',\ovf) \le \sum_{i=1}^n \diam(f_i).
$$
Since $n+1\in C$ we have $(i,n+1)\in J(\ovf)$ for some $i$
and $\supp(f_i)\cap\supp(f_{n+1})$ contains some point $x\in B$. 
Hence, by the triangle inequality,
$$
d(y,z)\le d(y,x)+d(x,z)
\le\diam(C',\ovf)+\diam(f_{n+1}) = \sum_{i=1}^{n+1} \diam(f_i)
$$
for any $y\in \bigcup_{j=1}^n \supp(f_j)$ and $z\in\supp(f_{n+1})$.
Therefore, the claimed inequality holds.
\end{proof}

\begin{lemma}
$\CC(\ovf)$ is a non-cruder partition of $\{1,\ldots,m\}$ than
$\CC_\varphi(\ovf)$.
\end{lemma}

\begin{proof}
Follows immediately from the fact that $J(\ovf) \subseteq J_\varphi(\ovf)$.
\end{proof}

\begin{lemma}
$\varphi(\supp(f_i)) \cap \varphi(\supp(f_j))=\varnothing$
$\ \Rightarrow\ $
$\supp(\varphi^\ast(f_i)) \cap \supp(\varphi^\ast(f_j))=\varnothing$.
\end{lemma}

\begin{proof}
Because 
$\supp(\varphi^\ast(f_i))\subseteq \varphi(\supp(f))$
for every $f\in A^B$.
\end{proof}

\subsection{Clusters associated with a solution for \eqref{eq:qsp_shift}}

Consider a solution $\ovdelta,N$ for \eqref{eq:qsp_shift}. 
The clusters $\CC_\varphi(\ovfdelta)$ corresponding to $\ovdelta,N$
are independent parts of the solution in the sense of the following lemma.
 
\begin{lemma}\label{le:clustersum}
$\varphi^\ast\rb{\sum_{i\in C} f_i^{\delta_i}}=_{A^{B/N}}0$ for every $C\in \CC_\varphi(\ovfdelta)$.
\end{lemma}

\begin{proof}
Since $\ovdelta, N$ is a solution for 
\eqref{eq:qsp_shift}, we have
$$
0 =
\varphi^\ast \left(\sum_{i=1}^m f_i^{\delta_i}\right) =
\sum_{i=1}^m \varphi^\ast(f_i^{\delta_i}) =
\sum_{C\in \CC_\varphi(\ovfdelta)}
\left(
\sum_{i\in C} \varphi^\ast(f_i^{\delta_i})
\right)
\mbox{ in } A^{B/N}.
$$
Clearly, for every $C\in \CC_\varphi(\ovfdelta)$,
$$
\supp\bigg(\sum_{i\in C} \varphi^\ast(f_i^{\delta_i})\bigg)
\subseteq
\bigcup_{i\in C} \supp( \varphi^\ast(f_i^{\delta_i})).
$$
Furthermore, by definition of a cluster, for 
distinct $C,D \in \CC_\varphi(\ovfdelta)$
we have
$$
\bigcup_{i\in C} \supp( \varphi^\ast(f_i^{\delta_i}))
\ \ \cap \ \ 
\bigcup_{i\in D} \supp( \varphi^\ast(f_i^{\delta_i}))=
\varnothing.
$$
Therefore, 
$\sum_{i\in C} \varphi^\ast(f_i^{\delta_i})=_{A^{B/N}}0$
for every 
$C \in \CC_\varphi(\ovfdelta)$.
\end{proof}

Below we consider $\ovdelta$ as an element of the abelian group $B^m$.
For $\Delta\in B$ and $C\subseteq \{1,\dots,m\}$ define
$\Delta^{(C)}\in B^m$ as $(b_1,\dots,b_m)$, where
$$
b_i=
\begin{cases}
\Delta& \mbox{if } i\in C\\
0& \mbox{if } i\notin C.
\end{cases}
$$
The next lemma claims
that shifting all functions from the same cluster simultaneously by $\Delta$ produces a new solution for \eqref{eq:qsp_shift}.

\begin{lemma}\label{le:cluster-shift}
If $(\ovdelta,N)$ satisfies \eqref{eq:qsp_shift}, 
then $(\ovdelta+\Delta^{(C)},N)$ satisfies \eqref{eq:qsp_shift}
for any $\Delta\in B$ and $C \in \CC_\varphi(\ovfdelta)$.
\end{lemma}

\begin{proof}\belowdisplayskip=-12pt
Since $(\ovdelta, N)$ satisfies \eqref{eq:qsp_shift}, 
it follows from Lemma \ref{le:clustersum} that
\begin{align*}
0 & =
\varphi^\ast \bigg(\sum_{i=1}^m f_i^{\delta_i}\bigg) \\
& = \varphi^\ast\left(\sum_{i\not\in C} f_i^{\delta_i}\right) + \varphi^\ast\left(\sum_{i\in C} f_i^{\delta_i}\right) \\
& \stackrel{\ref{le:clustersum}}{=}
\varphi^\ast\left(\sum_{i\not\in C} f_i^{\delta_i}\right) + \varphi^\ast\left(\sum_{i\in C} f_i^{\delta_i+\Delta^{(C)}}\right).
\end{align*}
\end{proof}

Our next goal is to prove Corollary \ref{cor:same_clusters}, which claims that every solution $(\ovdelta, N)$ for \eqref{eq:qsp_shift}
can be modified into a solution  $(\ovdelta', N)$
satisfying the additional condition
$\CC(\ovfdelta') = \CC_\varphi(\ovfdelta')$.

\begin{lemma}\label{le:cluster_equal}
Suppose that 
$\CC_\varphi(\ovfdelta) = \{C\}$, i.e. $C$ is a single cluster for functions $\varphi^\ast(\ovfdelta) \in A^{B/N}$.
Then there exists $\ovdelta' = (\delta_1',\ldots,\delta_m')$
satisfying the following:
\begin{itemize}
\item[(a)]
$\ovdelta'\equiv_N \ovdelta$;
\item[(b)]
$\sum_{i=1}^m f_i^{\delta_i}=_{A^{B/N}}\sum_{i=1}^m f_i^{\delta_i'}$;
\item[(c)]
$\CC(\ovf{}^{\ovdelta'})=\{C\}$.
\end{itemize}
\end{lemma}

\begin{proof}
We prove existence of such a sequence using induction on $m$. First, observe that (a) immediately implies (b). If $m=1$, then $C=\{1\}$ and all conditions are trivially satisfied for $\ovdelta'=\ovdelta$.
Assume that the statement holds for $m=n$.

We now address the case where $m=n+1$. Consider $n+1$ functions 
$f_1^{\delta_1},\ldots,f_{n+1}^{\delta_{n+1}}\in A^B$ where the functions $\{\varphi^\ast(f_i^{\delta_i})\}_{i=1}^{n+1}$ form a single cluster in $A^{B/N}$. 
Then the cluster-graph modulo $N$ for $\ovf$, $G_\varphi(\ovf)$, is connected
and has at least two vertices and, hence, contains at least one non-cut-vertex.
Without loss of generality, we may assume that the vertex $n+1$ is not a cut-vertex. 
Deleting $n+1$ from $G_\varphi(\ovf)$ results in a connected graph.
That means that the functions $\{\varphi^\ast(f_i^{\delta_i})\}_{i=1}^{n}$ form a single cluster in $A^{B/N}$ and the induction hypothesis applies
to $\{f_i^{\delta_i}\}_{i=1}^n$. Thus, there exist $\delta_1',\ldots,\delta_n'\in B$ satisfying (a), (b), and (c), which means that $\{f_i^{\delta_i'}\}_{i=1}^n$ are all in one cluster $C'$ in $A^B$ and $|\CC(\ovf{}^{\ovdelta'})|\leq 2$, where $\ovdelta'=(\delta_1',\ldots,\delta_n',0)$.

If $|\CC(\ovf{}^{\ovdelta'})| = 1$, then we are done.
Consider the case when $|\CC(\ovf{}^{\ovdelta'})| = 2$, i.e., when $n+1\notin C'$.
\begin{align*}
C=\{1,\ldots,n,n+1\}
&\ \ \Rightarrow\ \ 
(i,n+1)\in J_\varphi(\ovf{}^{\ovdelta'}) \mbox{ for some } 1\leq i\le n\\
&\ \ \Rightarrow\ \ 
\varphi(\supp\bigl(f_i^{\delta_i'})\bigr)\cap\varphi(\supp\bigl(f_{n+1}^{\delta_{n+1}})\bigr)
\ne\varnothing\\
&\ \ \Rightarrow\ \ 
\exists \ p\in \varphi(\supp\bigl(f_i^{\delta_i'})\bigr)\cap\varphi(\supp\bigl(f_{n+1}^{\delta_{n+1}})\bigr)\\
&\ \ \Rightarrow\ \ 
p=\varphi(p_i)=\varphi(p_{n+1}) \mbox{ for some } 
p_{n+1}\in\supp(f_{n+1}^{\delta_{n+1}}), \ p_i\in\supp(f_i^{\delta_i'}).
\end{align*}
Then for $\delta_{n+1}'=\delta_{n+1}+(p_i-p_{n+1})$
the following holds:
$$
\delta_{n+1}'\equiv_{N} \delta_{n+1}
\ \ \mbox{ and }\ \ 
p_i\in\supp(f_i^{\delta_i'})\cap\supp(f_{n+1}^{{\delta_{n+1}'}}).
$$
The former implies condition (a) and the latter
implies that $n+1$ belongs to the same cluster
as $i$ in $A^B$, i.e., 
condition (c) is satisfied.
\end{proof}

\begin{lemma}\label{le:same_clusters}
Let $\CC_\varphi(\ovfdelta)=\{C_1,\ldots,C_n\}$ for $n\geq 2$.
Then there exists $\ovdelta'$ satisfying
\begin{itemize}
\item[(a)]
$\ovdelta'\equiv_N \ovdelta$;
\item[(b)]
$\sum_{i=1}^m f_i^{\delta_i}=_{A^{B/N}}\sum_{i=1}^m f_i^{\delta_i'}$;
\item[(c)]
$\CC_\varphi(\ovf{}^{\ovdelta'})=\CC(\ovf{}^{\ovdelta'})$.
\end{itemize}
\end{lemma}

\begin{proof}
The transformation of $\delta_i$ from the proof of Lemma \ref{le:cluster_equal} 
applied to a single cluster does not change other clusters.
Consider the cluster $C_1\in\CC_\varphi(\ovfdelta)$. By \ref{le:cluster_equal}, 
there exists a sequence $\ovdelta_1'$ satisfying (a) and (b) such that $C_1\in\CC(\ovf{}^{\ovdelta_1'})$. 
We can proceed thus, applying \ref{le:cluster_equal} individually to each cluster $C_1,\ldots,C_n$. After termination of this process, we are guaranteed a sequence $\ovdelta_n'$ satisfying (a) and (b) such that $\CC(\ovf{}^{\ovdelta_n'})=\CC_\varphi(\ovf{}^{\ovdelta_n'})$, hence satisfying condition (c).
\end{proof}

\begin{cor}\label{cor:same_clusters}
For every solution $(\ovdelta, N)$ for \eqref{eq:qsp_shift},
there exists $\ovdelta'$ such that
\begin{itemize}
\item[(a)]
$(\ovdelta', N)$  satisfy \eqref{eq:qsp_shift};
\item[(b)]
$\CC(\ovf{}^{\ovdelta'})=\CC_\varphi(\ovf{}^{\ovdelta'})$.
\end{itemize}
\end{cor}

\begin{proof}
This is a special case of Lemma \ref{le:same_clusters},
where $\sum_{i=1}^m f_i^{\delta_i}=\sum_{i=1}^m f_i^{\delta_i'}=_{A^{B/N}}0$.
\end{proof}

\begin{lemma} \label{le:diam_size}
For any $f\in A^B$, $\diam(f)\le \size(f)$.
\end{lemma}

\begin{proof}
Follows directly from the definition of $\size(f)$.
\end{proof}

\begin{lemma}\label{le:bounded_delta_size}
For any $(\ovdelta, N)$ satisfying \eqref{eq:qsp_shift}, 
there exists $\ovdelta'$ such that
\begin{itemize}
\item[(a)] 
$(\ovdelta', N)$ satisfies \eqref{eq:qsp_shift};
\item[(b)] 
$|\delta_i'|\le \sum_{i=1}^m \size(f_i)$.
\end{itemize}
\end{lemma}

\begin{proof}
By Corollary \ref{cor:same_clusters}, 
we may assume that 
$\CC(\ovfdelta)=\CC_\varphi(\ovfdelta) = \{C_1,\dots,C_n\}$. 
For each $i=1,\dots,n$ choose an arbitrary 
$$
\Delta_i\in\bigcup_{j\in C_i}\supp\bigl(f_j^{\delta_{j}}\bigr)
$$
and define
$\ovdelta'=\ovdelta - \sum_{i=1}^k \Delta_i^{(C_i)}.$ 
By Lemma \ref{le:cluster-shift}, $(\ovdelta',N)$ satisfies
\eqref{eq:qsp_shift} and by construction
\begin{equation}\label{eq:zero-in-cluster}
0\in\bigcup_{j\in C_i}\supp\bigl(f_j^{\delta_{j}'}\bigr)
\end{equation}
for every $i=1,\dots,n$.
Therefore, for every $j\in C_i$ we have
$$\belowdisplayskip=-12pt
|\delta_j'| 
\stackrel{\eqref{eq:zero-in-cluster}}{\le}
\diam\bigl(C_i,\ovf{}^{\ovdelta'}\bigr) 
\stackrel{\ref{le:diameterbounded}}{\le} 
\sum_{i\in C}\diam(f_i^{\delta_i}) =
\sum_{i\in C}\diam(f_i) 
\stackrel{\ref{le:diam_size}}{\le}\sum_{i\in C}\size(f_i).
$$
\end{proof}

In the next lemma we use two subgroups $N,N'\unlhd B$ and
the corresponding functions
$\varphi_N^\ast: A^B\to A^{B/N}$ and
$\varphi_{N'}^\ast: A^B\to A^{B/N'}$.

\begin{lemma}\label{le:good_subgroup}
Let $f\in A^B$, $N\le B$, $S=\Set{p_i-p_j}{p_i,p_j\in\supp(f)}\subseteq B$, 
and $N'=\gp{N\cap S}\le N$. Then
$$
f=_{A^{B/N}}0 \ \ \Leftrightarrow \ \ f=_{A^{B/N'}}0.
$$
\end{lemma}

\begin{proof}
Observe that for any $p_i,p_j\in\supp(f)$,
\begin{align*}
p_i=_{B/N} p_j \ \ & \Leftrightarrow \ \ p_i-p_j=_{B/N} 0 \\
& \Leftrightarrow \ \ p_i-p_j\in N \\
& \Leftrightarrow \ \ p_i-p_j\in N' && \mbox{(since $p_i-p_j\in S$)}\\
& \Leftrightarrow \ \ p_i=_{B/N'}p_j.
\end{align*}
Therefore,
$\varphi_N^\ast(f)(p+N)=\varphi_{N'}^\ast(f)(p+N')$ for any $p\in \supp(f)$. 
Since $\varphi_N^\ast(f)$ is trivial, $\varphi_{N'}^\ast(f)$ must be trivial too.
\end{proof}

\begin{theorem}\label{th:qsp_np}
$\QSP(\cdot,\cdot,\cdot,\cdot)$ is in $\NP$.
\end{theorem}

\begin{proof}
Let $I=(A,B,\ovf,h)$ be a positive instance of $\QSP$. 
Then \eqref{eq:qsp_shift} has a solution $(\ovdelta,N)$
for some $\ovdelta=(\delta_1,\dots,\delta_m)$.
By Lemma \ref{le:bounded_delta_size}, we may assume that
$|\delta_i|\le \sum_{i=1}^m \size(f_i)$
for every $i=1,\dots,m$.
For the function $f=\sum_{i=1}^m f_i^{\delta_i}\in A^B$ 
define a set
$$
S=\Set{p_i-p_j}{p_i, p_j\in\supp(f)} \subseteq B,
$$
and a subgroup $N'=\gp{S\cap N}$.
Notice that $\rank(N')\le\rank(N)\le h$ because $N'\le N$ and
\begin{align*}
(\ovdelta,N) \mbox{ satisfies \eqref{eq:qsp_shift}} 
\ \Rightarrow\ 
f=_{A^{B/N}}0
\stackrel{\ref{le:good_subgroup}}{\ \Rightarrow\ }
f=_{A^{B/N'}}0
\ \Rightarrow\ 
(\ovdelta,N') \mbox{ satisfies \eqref{eq:qsp_shift}}.
\end{align*}
Furthermore, $|p|\le\size(f)$ for every $p\in S$, since $p=p_i-p_j$ 
for some $p_i,p_j\in\supp(f)$. 
Thus, $N'$ has a generating set $S\cap N$ that consists of short elements of $B$.
Now it is easy to see that
$(\ovdelta,N')$ is an $\NP$-certificate for $I$
because both parts of \eqref{eq:qsp_shift}, namely
$f=_{A^{B/N'}}0$ and $\rank(N')\le h$
can be checked in polynomial time
(see Section \ref{se:rank}).
\end{proof}

\subsection{Quotient-rank bound}
\label{se:quotient-rank-bounds}

By Theorem \ref{th:QSP-infinite-B}, 
$\QSP(A,B,\cdot,h=0)$ is $\NP$-hard when $\MZ\le B$.
In this section we investigate hardness of $\QSP$
when the quotient-rank bound $h$ is greater than $0$.
Recall that by $\rank(B)$ we denote the minimum number of generators for $B$.
First, we show that the problem is efficiently decidable
when the value of the quotient-rank bound $h$ is high.

\begin{theorem}\label{th:big-h}
$\QSP(\cdot,B,\cdot,\cdot_{h\ge \rank(B)}) \in\P$.
\end{theorem}

\begin{proof}
Consider an instance of $I=(A,B,\ovf,h)$ of $\QSP$, where 
$\ovf=(f_1,\dots,f_m)$ and $h\ge \rank(B)$.
We claim that
$$
I \mbox{ has a solution }
\ \ \Leftrightarrow\ \ 
\sum_{i=1}^m \sum_{x\in\supp(f_{i})} f_{i}(x)=0.
$$
Clearly, the latter condition can be checked in polynomial time.

``$\Rightarrow$'' 
Follows from Lemma \ref{le:trivial-sum}.

``$\Leftarrow$'' 
Since $h\ge \rank(B)$ we can choose the quotient $N=B$.
In that case the domain of $\sum_{i=1}^m f_i^{\delta_i}$ is 
$B/N=\{1\}$ and its single value is exactly
$\sum_{i=1}^m \sum_{x\in\supp(f_{i})} f_{i}(x)=0$.
\end{proof}

Next we show that the bound $h\ge\rank(B)$ cannot be improved, 
i.e., the problem is $\NP$-hard when $h=\rank(B)-1$
for some groups $B$. 
To do so we generalize the proof of Proposition \ref{pr:qsp_tpart}
and relate instances of $\TPART$ to instances of $\QSP$.
Fix $A=\MZ$ and $B=\MZ^{h}$. 
Let $T=\{t_1,\ldots,t_{3k}\}$ be an instance of $\TPART$
and $L = \frac{1}{k}\sum_{i=1}^{3k} t_i$ be the target sum for subsets.
Let $M=k(L+1)$, which is clearly greater than $\sum_{i=1}^{3k} t_i$.
Fix the standard basis $b_1,\dots,b_{h}$ for $B=\MZ^{h}$.
Let $a=1\in A=\MZ$. Define
\begin{equation}\label{eq:mid-h}
c_y = \sum_{i=0}^{y-1} a^{-i\cdot b_{h}} \mbox{ (for $y\in\MN$)} \ \ 
\mbox{ and }\ \ 
c=\sum_{i=0}^{k-1} 
c_L^{-i(L+1)\cdot b_h}
+
\sum_{i=1}^{h-1} 2^{i}M(a^{b_{i}} - a^{2\cdot b_{i}})
\end{equation}
Let $\ovc=(c_{t_1},\ldots,c_{t_{3k}},-c)$.
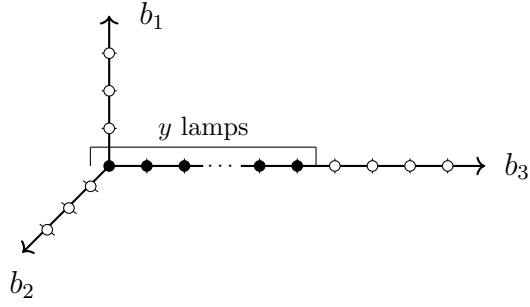
\begin{figure}[!h]
\centering
\begin{tikzpicture}
    \coordinate (O) at (0,0,0);
    \draw[thick,-] (0,0,0) -- (1.25,0,0);
  \draw[thick,->] (1.75,0,0) -- (5,0,0) node[left=-0.75]{$b_3$};
  \draw[thick,->] (0,0,0) -- (0,2,0) node[right=0.25]{$b_1$};
  \draw[thick,->] (0,0,0) -- (0,0,3) node[above=-0.75]{$b_2$};

  \draw[-] (0.5,-0.1)--(0.5,0.1);
 \draw[-] (1,-0.1)--(1,0.1);
 \draw[-] (2,-0.1)--(2,0.1);
 \draw[-] (2.5,-0.1)--(2.5,0.1);
 \draw[-] (3,-0.1)--(3,0.1);
 \draw[-] (3.5,-0.1)--(3.5,0.1);
 \draw[-] (4,-0.1)--(4,0.1);
 \draw[-] (4.5,-0.1)--(4.5,0.1);

 \draw[-] (-0.1,0.5,0)--(0.1,0.5,0);
 \draw[-] (-0.1,1,0)--(0.1,1,0);
 \draw[-] (-0.1,1.5,0)--(0.1,1.5,0);

 \draw[-] (-0.25,-0.1,0.25)--(-0.05,-0.25,0.25);
 \draw[-] (-0.25,-0.1,1)--(-0.05,-0.25,1);
 \draw[-] (-0.25,-0.1,1.75)--(-0.05,-0.25,1.75);

 \filldraw[draw=black,fill=black] (0,0,0) circle (2pt);
 \filldraw[draw=black,fill=black] (0.5,0,0) circle (2pt);
 \filldraw[draw=black,fill=black] (1,0,0) circle (2pt);
 \filldraw[draw=black,fill=black] (2,0,0) circle (2pt);
 \filldraw[draw=black,fill=black] (2.5,0,0) circle (2pt);
 \filldraw[draw=black,fill=white] (3,0,0) circle (2pt);
 \filldraw[draw=black,fill=white] (3.5,0,0) circle (2pt);
 \filldraw[draw=black,fill=white] (4,0,0) circle (2pt);
 \filldraw[draw=black,fill=white] (4.5,0,0) circle (2pt);

 \filldraw[draw=black,fill=white] (0,0.5,0) circle (2pt);
 \filldraw[draw=black,fill=white] (0,1,0) circle (2pt);
 \filldraw[draw=black,fill=white] (0,1.5,0) circle (2pt);

 \filldraw[draw=black,fill=white] (-0.15,-0.175,0.25) circle (2pt);
 \filldraw[draw=black,fill=white] (-0.15,-0.175,1) circle (2pt);
 \filldraw[draw=black,fill=white] (-0.15,-0.175,1.75) circle (2pt);

        \draw [-] (-0.25,0,0) to (-0.25,0.25,0);
        \draw [-] (2.75,0,0) to (2.75,0.25,0);
        \draw [-] (-0.25,0.25,0) to (2.75,0.25,0);

        \node[draw=none] at (1.5,0) {\tiny$\cdots$};
        \node[draw=none] at (1.25,0.5,0) {\footnotesize{$y$ lamps}};
\end{tikzpicture}
\caption{Schematic picture for the lamp configuration in $c_y$ defined by \eqref{eq:mid-h}, where $h=3$.}
\end{figure}

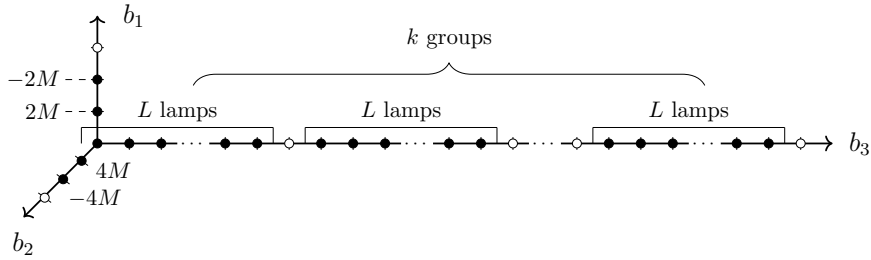
\begin{figure}[!h]
\centering
\scalebox{0.85}{
\begin{tikzpicture}
    \coordinate (O) at (0,0,0);
    \draw[thick,-] (0,0,0) -- (1.25,0,0);
    \draw[thick,-] (1.75,0,0) -- (4.75,0,0);
    \draw[thick,-] (5.25,0,0) -- (6.75,0,0);
    \draw[thick,-] (7.25,0,0) -- (9.25,0,0);
  \draw[thick,->] (9.75,0,0) -- (11.5,0,0) node[left=-0.75]{$b_3$};
  \draw[thick,->] (0,0,0) -- (0,2,0) node[right=0.25]{$b_1$};
  \draw[thick,->] (0,0,0) -- (0,0,3) node[above=-0.75]{$b_2$};

  \draw[-] (0.5,-0.1)--(0.5,0.1);
 \draw[-] (1,-0.1)--(1,0.1);
 \draw[-] (2,-0.1)--(2,0.1);
 \draw[-] (2.5,-0.1)--(2.5,0.1);
 \draw[-] (3,-0.1)--(3,0.1);
 \draw[-] (3.5,-0.1)--(3.5,0.1);
 \draw[-] (4,-0.1)--(4,0.1);
 \draw[-] (4.5,-0.1)--(4.5,0.1);
 \draw[-] (5.5,-0.1)--(5.5,0.1);
 \draw[-] (6,-0.1)--(6,0.1);
 \draw[-] (6.5,-0.1)--(6.5,0.1);
 \draw[-] (7.5,-0.1)--(7.5,0.1);
 \draw[-] (8,-0.1)--(8,0.1);
 \draw[-] (8.5,-0.1)--(8.5,0.1);
 \draw[-] (9,-0.1)--(9,0.1);
 \draw[-] (10,-0.1)--(10,0.1);
 \draw[-] (10.5,-0.1)--(10.5,0.1);
 \draw[-] (11,-0.1)--(11,0.1);

 \draw[-] (-0.1,0.5,0)--(0.1,0.5,0);
 \draw[-] (-0.1,1,0)--(0.1,1,0);
 \draw[-] (-0.1,1.5,0)--(0.1,1.5,0);

 \draw[-] (-0.25,-0.1,0.25)--(-0.05,-0.25,0.25);
 \draw[-] (-0.25,-0.1,1)--(-0.05,-0.25,1);
 \draw[-] (-0.25,-0.1,1.75)--(-0.05,-0.25,1.75);

 \filldraw[draw=black,fill=black] (0,0,0) circle (2pt);
 \filldraw[draw=black,fill=black] (0.5,0,0) circle (2pt);
 \filldraw[draw=black,fill=black] (1,0,0) circle (2pt);
 \filldraw[draw=black,fill=black] (2,0,0) circle (2pt);
 \filldraw[draw=black,fill=black] (2.5,0,0) circle (2pt);
 \filldraw[draw=black,fill=white] (3,0,0) circle (2pt);
 \filldraw[draw=black,fill=black] (3.5,0,0) circle (2pt);
 \filldraw[draw=black,fill=black] (4,0,0) circle (2pt);
 \filldraw[draw=black,fill=black] (4.5,0,0) circle (2pt);
 \filldraw[draw=black,fill=black] (5.5,0,0) circle (2pt);
 \filldraw[draw=black,fill=black] (6,0,0) circle (2pt);
 \filldraw[draw=black,fill=white] (6.5,0,0) circle (2pt);
 \filldraw[draw=black,fill=white] (7.5,0,0) circle (2pt);
 \filldraw[draw=black,fill=black] (8,0,0) circle (2pt);
 \filldraw[draw=black,fill=black] (8.5,0,0) circle (2pt);
 \filldraw[draw=black,fill=black] (9,0,0) circle (2pt);
 \filldraw[draw=black,fill=black] (10,0,0) circle (2pt);
 \filldraw[draw=black,fill=black] (10.5,0,0) circle (2pt);
 \filldraw[draw=black,fill=white] (11,0,0) circle (2pt);

 \filldraw[draw=black,fill=black] (0,0.5,0) circle (2pt);
 \filldraw[draw=black,fill=black] (0,1,0) circle (2pt);
 \filldraw[draw=black,fill=white] (0,1.5,0) circle (2pt);

 \filldraw[draw=black,fill=black] (-0.15,-0.175,0.25) circle (2pt);
 \filldraw[draw=black,fill=black] (-0.15,-0.175,1) circle (2pt);
 \filldraw[draw=black,fill=white] (-0.15,-0.175,1.75) circle (2pt);

        \draw [-] (-0.25,0,0) to (-0.25,0.25,0);
        \draw [-] (2.75,0,0) to (2.75,0.25,0);
        \draw [-] (-0.25,0.25,0) to (2.75,0.25,0);

        \draw [-] (-0.25+3.5,0,0) to (-0.25+3.5,0.25,0);
        \draw [-] (2.75+3.5,0,0) to (2.75+3.5,0.25,0);
        \draw [-] (-0.25+3.5,0.25,0) to (2.75+3.5,0.25,0);

        \draw [-] (-0.25+8,0,0) to (-0.25+8,0.25,0);
        \draw [-] (2.75+8,0,0) to (2.75+8,0.25,0);
        \draw [-] (-0.25+8,0.25,0) to (2.75+8,0.25,0);

        \node[draw=none] at (1.5,0) {\tiny$\cdots$};
        \node[draw=none] at (5,0) {\tiny$\cdots$};
        \node[draw=none] at (7,0) {\tiny$\cdots$};
        \node[draw=none] at (9.5,0) {\tiny$\cdots$};
        \node[draw=none] at (1.25,0.5,0) {\footnotesize{$L$ lamps}};
        \node[draw=none] at (1.25+3.5,0.5,0) {\footnotesize{$L$ lamps}};
        \node[draw=none] at (1.25+8,0.5,0) {\footnotesize{$L$ lamps}};

        \node[draw=none] at (-0.85,0.5,0) {\footnotesize{$2M$}};
        \node[draw=none] at (-1,1,0) {\footnotesize{$-2M$}};
        \draw[dashed] (-0.5,1,0)-- (0,1,0);
        \draw[dashed] (-0.5,0.5,0)-- (0,0.5,0);

        \node[draw=none] at (-0.15+0.5,-0.175-0.15,0.25) {\footnotesize{$4M$}};
        \node[draw=none] at (-0.15+0.5,-0.175-0.3,1) {\footnotesize{$-4M$}};
        \draw [decorate,decoration={brace,amplitude=10pt,raise=5ex}]
  (1.5,0,0) -- (9.5,0,0) node[midway,yshift=4em]{\footnotesize{$k$ groups}};
	\end{tikzpicture}
}
\caption{Schematic picture for the lamp configuration 
in $c$ defined by \eqref{eq:mid-h}, where $h=3$.}
\label{fi:element-c}
\end{figure}

\begin{proposition}\label{pr:mid-h}
$T$ is a positive instance of $\TPART$ if and only if 
$(\MZ,\MZ^h,\ovc,h-1)$ is a positive instance of $\QSP$.
\end{proposition}

\begin{proof}
``$\Rightarrow$''
If $T$ is a positive instance of $\TPART$, then, without loss of generality, 
we may assume that
$$
\sum_{i=1}^3 t_{3j+i}=L \ \ \text{for} \ \ j=0,1,\ldots,k-1.
$$
Let $L_i=(L+1)i$ for $i=0,\dots,k-1$.
Then clearly $N=\gp{b_1,\dots,b_{h-1}}$ and
\begin{align*}
\delta_1&=0 &\delta_2&=-t_1 &\delta_3&=-t_1-t_2\\
\delta_4&=-L_1 &\delta_5&=-L_1-t_4 &\delta_6&=-L_1-t_4-t_5\\
\delta_7&=-L_2 &\delta_8&=-L_2-t_7 &\delta_9&=-L_2-t_7-t_8\\
&\cdots &&\cdots &&\cdots 
\end{align*}
satisfy \eqref{eq:qsp_shift}.

``$\Leftarrow$''
Suppose that
$(\MZ,\MZ^h,\ovc,h-1)$ is a positive instance of $\QSP$, and
$N\le \MZ^h$ and shifts $\delta_1,\dots,\delta_{3k}$
satisfy \eqref{eq:qsp_shift}.
We claim that in the quotient space $\MZ^h/N$
each term $2^{i}M(a^{b_{i}} - a^{2\cdot b_{i}})$ in $c$ must vanish.
Indeed, every such term ``lights up'' two lamps 
\begin{itemize}
\item 
the lamp $2^{i}M a\in \MZ$ at $b_{i}\in \MZ^h$,
\item 
the lamp $-2^{i}M a\in \MZ$ at $2\cdot b_{i}\in \MZ^h$.
\end{itemize}
It is easy to see that $b_{h-1}$ and $2\cdot b_{h-1}$ must collapse in $\MZ^h/N$
(otherwise the lamp-values $2^{h-1} M$ and $-2^{h-1} M$ will not vanish
-- the rest of the lamp-values are too small to make $2^{h-1} M$ and $-2^{h-1} M$
disappear and satisfy \eqref{eq:qsp_shift}).
That implies that $b_{h-1}\in N$.
Then the statement follows by induction (from $h-1$ to $1$).
Thus, $N=\gp{b_1,\dots,b_{h-1}}$ and, modulo $N$, we have the situation 
discussed in the proof of Proposition \ref{pr:qsp_tpart}.
\end{proof}

\begin{theorem}\label{th:mid-h}
Fix $h\in \MN$.
$\QSP(\MZ,\MZ^{h},\cdot,h-1)$ is $\NP$-complete.
\end{theorem}

\begin{proof}
Follows from Proposition \ref{pr:mid-h} and Theorem \ref{th:qsp_np}.
\end{proof}

Imposing a strict condition $|\ovf|=1$ on the number of functions in $\ovf$
allows to get a uniform result over the choice of $A$ and $B$.

\begin{theorem}\label{th:malcev}
Fix $h\in \MN$.
$\QSP(\cdot,\cdot,\cdot_{|\ovf|=1},h)$ is in $\P$.
\end{theorem}

\begin{proof}
Suppose that a given instance $I=(A,B,\ovf,h)$ of $\QSP$ has a solution, 
where $\ovf=(f_1)$.
Then $f_1^{\delta_1}=_{A^{B/N}} 0$ for some $N\le B$ satisfying 
$\rank(N)\le h$. 
For simplicity we assume that $B=\MZ^n$.
Clearly, we may assume that  $\delta_1=0$.
Below we prove that the description of a suitable
subgroup $N$ can be found in polynomial time.

As in the proof of Theorem \ref{th:qsp_np}, we can define the set 
$$
S=\Set{p_i-p_j}{p_i, p_j\in\supp(f_1)} \subseteq B,
$$
and the subgroup $N'=\gp{S\cap N} \le N$, which is generated by 
some elements of $S$
(in general, by more than $h$ elements of $S$)
and that works as $N$, i.e.,
$$
f_1=_{A^{B/{N'}}} 0
\ \ \mbox{ and }\ \ 
\rank(N') \le \rank(N) \le h.
$$
Fix any basis for $N'=\gp{\ovv_1,\dots,\ovv_h}$
(notice that in general $\ovv_i\notin S$) and define a new subgroup
\begin{equation}\label{eq:no-hole}
N''=N''(\ovv_1,\dots,\ovv_h)=
\Set{\alpha_1\ovv_1+\dots+\alpha_h\ovv_h \in B}{\alpha_i\in\MR}.
\end{equation}
Observe that $N''$ is the maximum subgroup of $B$
containing $N'$ and satisfying
$\rank(N'')=\rank(N')$, i.e., $N''$ works as $N'$.
Notice that \eqref{eq:no-hole} defines $N''$
as an intersection of the vector space
$\Span(\ovv_1,\dots,\ovv_h)$ and $B=\MZ^n$.
Furthermore, $\ovv_1,\dots,\ovv_h \in N'=\gp{S\cap N}$.
Hence, $N''$ can be captured by choosing up to $h$
elements $\ovs_1,\ovs_2,\dots$ from $S$ as 
\begin{equation}\label{eq:no-hole2}
N''=
\Set{\alpha_1 \ovs_1+\alpha_2\ovs_2+\dots \in B}{\alpha_i\in\MR},
\end{equation}
because for that purpose it is sufficient to capture the underlying vector space.
Therefore, to find $N''$ one can enumerate all tuples of at most $h$
elements $(\ovs_1,\ovs_2,\dots)$ from $S$, 
and for each tuple check if $f_1 =_{A^{B/{N''}}} 0$.
To check the latter condition, it is not even necessary
to find a generating set for $N''$, one can work modulo 
$\Span(\ovs_1,\ovs_2,\dots)$, which can be done in polynomial time.

Finally, the number of tuples of length $h$ is ${|S| \choose h}=O(|S|^h)$ 
and $|S|=O(|\supp(f_1)|^2)$.
Hence, the number of all tuples
of length at most $h$ is $O(|\supp(f_1)|^{2h})$.
Thus, the described procedure can be performed in polynomial time.
\end{proof}

\subsection{Bounded number of functions}

Here we consider a version of $\QSP$ in which 
the number of given functions in $\ovf=(f_1,\dots,f_m)$
is bounded by a fixed constant $M$.

\begin{theorem}\label{th:bounded-m}
$\QSP(\cdot,B,\cdot_{|\ovf|<M},\cdot)$ is in $\P$
for any finitely generated abelian group $B$ and $M\in \MN$.
\end{theorem}

\begin{proof}
By Theorem \ref{th:big-h}, $\QSP(\cdot,B,\cdot,\cdot_{h\ge \rank(B)}) \in\P$.
Hence, we may assume that 
$$
h < \rank(B).
$$
Consider an instance $I=(A,B,\ovf,h)$ of $\QSP$, where
$\ovf=(f_1,\dots,f_m)$ satisfies $m\le M$.
By definition, $I$ is a positive instance if and only if there exists
$(\ovdelta,N)$ for $I$ satisfying \eqref{eq:qsp_shift}.
Below we describe an enumeration procedure
that finds a pairs $(\ovdelta,N)$ satisfying \eqref{eq:qsp_shift}
for every positive instances $I$ 
in time polynomial in terms of $\size(I)$.

By Lemma \ref{le:bounded_delta_size}, we may assume that 
$|\delta_{i}|\le \sum_{i=1}^m \size(f_{i}) \le \size(I)$. 
Since $B$ is an abelian group, its growth function is bounded by a
polynomial function $p(n)$.
Hence, the number of tuples $(\delta_1,\dots,\delta_m)$
satisfying $|\delta_{i}|\le \size(I)$
is bounded by $(p(\size(I)))^M$, which is a polynomial
in terms of $\size(I)$.
Therefore, we can enumerate all bounded tuples $\ovdelta=(\delta_1,\dots,\delta_m)$
one by one in polynomial time.

For a bounded sequence $\ovdelta$ define
\begin{equation}\label{eq:c-condition}
c=\sum_{i=1}^m f_{c_i}^{\delta_{i}}\in A^B.
\end{equation}
Assume that for $\ovdelta$ there exists a subgroup $N\le B$ such that
$(\ovdelta,N)$ satisfies \eqref{eq:qsp_shift}, i.e.,
$$
c=_{A^{B/N}}0 \mbox{ and } \rank(N)\le h.
$$
We claim that we can find such subgroup (perhaps not $N$ itself)
as follows.
Define a set $S\subseteq B$ as
$$
S = S_c = \Set{p_i-p_j}{p_i,p_j\in\supp(c)}.
$$
Let $S'=N\cap S$ and let $N'=\gp{S'}\le N$. 
By Lemma \ref{le:good_subgroup}, $c=_{A^{B/N'}}0$. 
Since $N'\le N$, we have $\rank(N')\le \rank(N)\le h$.
By construction, every $s\in S'$ satisfies $\|s\|\le\size(I)$. 
Then by Corollary \ref{cor:bounded_gen_set}, 
$N'$ has a basis $\{t_1,\ldots,t_y\}$ satisfying $y\le h<\rank(B)$ and
$$
\|t_i\|\le \sqrt{2^{y-1}}\size(I)\le 2^{\rank(B)/2}\size(I).
$$
Here $2^{\rank(B)/2}$ is a constant -- it is a characteristic of the fixed group $B$.
Therefore, the size of the set
$$
T=T_c=
\Set{(t_1,\ldots,t_y)}{\|t_i\|\le 2^{\rank(B)/2}\size(I)}
$$ 
is polynomial in terms of $\size(I)$
and we can enumerate its elements one-by-one in polynomial time.
One tuple $(t_1,\ldots,t_y)\in T$ is a generating set for $N'$.
We can identify it by checking \eqref{eq:c-condition},
which can be done in polynomial time.
\end{proof}


\section{Orientable equations over $A\wr B$}
\label{se:orientable}

In this section we use the computational properties of
$\QSP$ established in Section \ref{se:QSP-properties}
to draw conclusions about the 
Diophantine problem for orientable quadratic equations \eqref{eq:orientable}.
We consider different variations of the Diophantine problem:
the uniform problem (in which $A$ and $B$ are a part of the input)
and the non-uniform problem (in which $A$ and $B$ are fixed)
together with additional conditions that can be imposed
on a given quadratic equation.
We start with the global complexity upper bound
for the uniform Diophantine problem for orientable quadratic equations.

\begin{corollary}
\label{co:equations-NP}
The problem of deciding if a given orientable quadratic equation in a given wreath product $A\wr B$ has a solution belongs to $\NP$.
\end{corollary}

\begin{proof}
Proposition \ref{pr:DP-QSP-reduction} translates a given equation into an instance of $\QSP(\cdot,\cdot,\cdot,\cdot)$ which, by Theorem \ref{th:qsp_np}, is in $\NP$.
\end{proof}

\subsection{Non-uniform problem}

Fix finitely generated abelian groups $A$ and $B$. Here we discuss orientable equations over the fixed wreath product $A\wr B$.

\begin{corollary}\label{co:DP-infinite-B}
If $|B|=\infty$, then the Diophantine problem
for \textbf{spherical} equations in $A\wr B$ is $\NP$-hard.
\end{corollary}

\begin{proof}
Proposition \ref{pr:DP-QSP-reduction} establishes a one-to-one correspondence
between spherical equations \eqref{eq:spherical} in $A\wr B$
and instances of $\QSP(A,B,\cdot,0)$, which,
by Theorem \ref{th:QSP-infinite-B}, is $\NP$-hard when $|B|=\infty$.
\end{proof}

\begin{corollary}\label{co:all-orientable}
The Diophantine problem for orientable quadratic equations over 
a wreath product of finitely generated abelian groups $A\wr B$ is
\begin{itemize}
\item
polynomial-time decidable $\ \ \Leftrightarrow\ \ $ $B$ is finite;
\item
$\NP$-complete $\ \ \Leftrightarrow\ \ $ $B$ is infinite.
\end{itemize}
\end{corollary}

\begin{proof}
By Corollary \ref{co:equations-NP}, the problem is in $\NP$.
By Corollary \ref{co:DP-infinite-B}, the problem is $\NP$-hard when $|B|=\infty$.
Finally, Proposition \ref{pr:DP-QSP-reduction} reduces
orientable quadratic equations \eqref{eq:orientable} in $A\wr B$
to instances of $\QSP(A,B,\cdot,\cdot)$, which,
by Theorem \ref{th:QSP-finite-B}, are polynomial-time decidable if $|B|<\infty$.
\end{proof}

Corollary \ref{co:all-orientable} gives a very limited picture
for orientable equations.
For $\NP$-hardness it employs spherical equations only
and neglects other types of orientable equations.
Below we analyze other cases with the goal
to identify the reasons for computational hardness.

\subsubsection{Non-uniform problem: genus $g$}

The value of the genus $g$ of a given equation $\CE$
relative to the rank of the group $B$ directly affects the hardness of solving $\CE$.

\begin{corollary}\label{co:g-vs-rankB}
The Diophantine problem for the class of orientable quadratic 
equations \eqref{eq:orientable}
with genus satisfying $g\ge \rank(B)/2$
is in $\P$.
\end{corollary}

\begin{proof}
Proposition \ref{pr:DP-QSP-reduction} reduces
orientable quadratic equations \eqref{eq:orientable} in $A\wr B$
satisfying $g\ge \rank(B)/2$ to instances of $\QSP(A,B,\cdot,2g)$, which, by
Theorem \ref{th:big-h}, are decidable in polynomial time.
\end{proof}

The bound $g\ge \rank(B)/2$ cannot be improved, 
i.e., there is a group $B$ satisfying $2g<\rank(B)$ for which 
the problem is $\NP$-hard. 

\begin{corollary}\label{co:mid-h}
Fix $g,n\in \MN$ such that $2g<n$.
The Diophantine problem for the class of orientable quadratic 
equations \eqref{eq:orientable}
of genus $g$ over the group $\MZ\wr \MZ^{n}$
is $\NP$-complete.
\end{corollary}

\begin{proof}
Proposition \ref{pr:DP-QSP-reduction} reduces
instances of $\QSP(\MZ,\MZ^n,\cdot,2g)$, satisfying the condition $2g < n$,
to orientable quadratic equations \eqref{eq:orientable} in $\MZ\wr \MZ^n$
of genus $g$. By Theorem \ref{th:mid-h}, the original problem
$\QSP(\MZ,\MZ^n,\cdot,2g)$ is $\NP$-complete.
\end{proof}

\subsubsection{Non-uniform problem: number of conjugates $m$}

The number of conjugates $m$ in a given quadratic equation $\CE$ of type \eqref{eq:orientable} is also responsible for the hardness of solving $\CE$. Specifically, equations \eqref{eq:orientable} with
$m$ bounded by a constant $M$ can be solved efficiently.

\begin{corollary}\label{co:bounded-m}
Fix $M\in \MN$.
The Diophantine problem for the class of orientable quadratic 
equations \eqref{eq:orientable}
with the number of conjugates $m$ bounded by $M$
over $A\wr B$ is in $\P$.
\end{corollary}

\begin{proof}
Proposition \ref{pr:DP-QSP-reduction} reduces
orientable quadratic equations \eqref{eq:orientable} in $A\wr B$ satisfying the condition $m\le M$ to instances of
$\QSP(\cdot,B,\cdot_{|\ovf|<M},\cdot)$, which, by
Theorem \ref{th:bounded-m}, are decidable in polynomial time.
\end{proof}

\subsubsection{Non-uniform problem: group $A$}
\label{se:group-A}

The choice of the (nontrivial) group $A$ is irrelevant
for $\NP$-hardness of the Diophantine problem.

\begin{corollary}\label{co:irrelevant-A2}
For any nontrivial finitely generated abelian groups $A_1,A_2$, and $B$
the Diophantine problem over $A_1\wr B$ and
over $A_2\wr B$ are either both $\NP$-hard or are both in $\P$.
\end{corollary}

\begin{proof}
By Proposition \ref{pr:DP-QSP-reduction}, the Diophantine problems
over $A_1\wr B$ and over $A_2\wr B$ are polynomial-time
equivalent to $\QSP(A_1,B,\cdot,\cdot)$ and $\QSP(A_2,B,\cdot,\cdot)$
respectively.
Hence, the statement follows from Corollary \ref{co:irrelevant-A}.
\end{proof}

\subsection{Uniform problem}

By Corollaries \ref{co:equations-NP} and \ref{co:DP-infinite-B}, the uniform Diophantine problem for orientable equations is $\NP$-complete.

Here we make one observation about the contribution of the group $A$ to the computational hardness of solving spherical equations.
Even though a particular choice of $A$ does not affect the $\NP$-hardness of the problem as observed in Section \ref{se:group-A}, we can argue that different choices of $A$ make the problem easier or harder.

\begin{corollary}\label{co:finite-B}
The Diophantine problem for spherical equations 
over the class of groups $\{\MZ^n\wr \MZ_2\}_{n\in\MN}$
is $\NP$-hard.
\end{corollary}

\begin{proof}
Follows from Theorem \ref{th:B-Z2}.
\end{proof}

Another observation is that
Malcev-type equations in wreath products of finitely generated abelian groups can be solved in polynomial time.

\begin{corollary}\label{co:malcev}
Fix $g\in \MN$.
The Diophantine problem for equations of genus $g$
$$
\prod_{i=1}^g [x_i,y_i] = c
$$
over a given wreath product $A\wr B$ is in $\P$.
\end{corollary}

\begin{proof}
Follows from Theorem \ref{th:malcev}.
\end{proof}

\section{Conclusion}
\label{se:Conclusion}

Ultimately, the main parameter responsible for the computational 
hardness of the Diophantine problem for orientable equations \eqref{eq:orientable}
over wreath products $A\wr B$ is the number of conjugates $m$.
Imposing a bound on $m$ makes the problem computationally 
feasible (Corollary \ref{co:bounded-m}).
Certain values of the genus $g$, the size of $B$, and $\rank(B)$
can make the problem feasible as well 
(Corollaries \ref{co:all-orientable} and \ref{co:g-vs-rankB}).
In the uniform case, when $A$ and $B$ are a part of the input,
unbounded rank of $A$ makes the problem computationally hard
even for small finite $B$ (Corollary \ref{co:finite-B}).

\bibliography{main_bibliography}

\end{document}